\documentclass[a4paper,11pt]{article}
\usepackage[utf8]{inputenc}
\usepackage[margin=1.00in]{geometry}

\usepackage{amsmath} 
\usepackage{amsthm}
\usepackage{amsfonts} 
\usepackage{doi} 
\usepackage{graphicx}
\usepackage{xcolor} 

\usepackage{enumerate}
\usepackage{amssymb} 

\newtheorem{thm}{Theorem}[section]
\newtheorem{lem}[thm]{Lemma}
\newtheorem{prop}[thm]{Proposition}
\newtheorem{cor}[thm]{Corollary}

\newtheorem{rem}[thm]{Remark}
\newtheorem*{conf*}{Property}

\numberwithin{equation}{section}

\newcommand{\EE}{\mathbb{E}} 
\newcommand{\PP}{\mathbb{P}} 

\usepackage{comment}
\excludecomment{toexclude}
\usepackage{comment}

\usepackage{hyperref}
\hypersetup{
    colorlinks=true, 
    linktoc=all,     
    citecolor=green, 
    linkcolor=blue,  
}
\usepackage{bm} 
\usepackage{mathabx} 
\newcommand{\bh}[1]{\widehat{#1}} 
\newcommand{\bc}[1]{\widecheck{#1}} 

\title{Joint invariance principles for random walks with positively and negatively reinforced steps}
\author{Marco Bertenghi\footnote{marco.bertenghi@math.uzh.ch} \ and Alejandro Rosales-Ortiz\footnote{alejandro.rosalesortiz@math.uzh.ch}}
\date{Institute of Mathematics, University of Zurich}

\begin{document}

\maketitle

\abstract{ {\footnotesize Given a random walk $(S_n)$ with typical step distributed according to      some fixed law and a fixed parameter $p \in (0,1)$,  
    the associated positively step-reinforced random walk is a discrete-time process which performs at each step, with probability $1-p$, the same step as $(S_n)$  while with probability $p$, it repeats one of the steps it performed previously  chosen uniformly at random. The negatively step-reinforced random walk follows the same dynamics but when a step is repeated its sign is also changed. In this work, we shall prove functional limit theorems for the triplet of a random walk, coupled with its positive and negative reinforced versions when $p < 1/2$ and when the typical step is centred. The limiting process is Gaussian and admits a simple representation in terms of stochastic integrals, 
    \begin{equation*}
    \left( B_t , \,  t^p \int_0^t s^{-p} \mathrm{d}\beta^r_s , \,  t^{-p} \int_0^t s^{p} \mathrm{d}\beta^c_s \right)_{t \in \mathbb{R}^+}
    \end{equation*}
    for properly correlated Brownian motions $B, \beta^r$, $\beta^c$. The processes in the second and third coordinate are called the noise reinforced Brownian motion (as named in \cite{BertoinUniversality}), and the noise counterbalanced Brownian motion of $B$. Different couplings are also considered, allowing us in some cases to drop the centredness hypothesis and to completely identify for all regimes $p \in (0,1)$ the limiting behaviour of step reinforced random walks. Our method exhausts a martingale approach in conjunction with the martingale functional CLT.}
}
\begin{toexclude}
\abstract{In this article we consider two different types of reinforcement algorithms for random walks, positive and negative reinforcement, which fall in the framework of random walks with memory. We establish invariance principles for both the positive step-reinforced random walks $\hat{S}$ as well as the negative step-reinforced random walks $\check{S}$. For the former, we give an alternative proof of the invariance principle in the diffusive regime $p \in (0,1/2)$ due to Bertoin \cite{BertoinUniversality} and also establish a novel invariance principle for the critical regime $p=1/2$. The latter exhibits no phase transition and it always behaves diffusively. Our method exhausts a martingale approach in conjunction with the martingale functional CLT. The laws of the  limiting processes obtained can always be expressed in terms of a member of the one parameter family of stochastic integrals  
\begin{equation} \label{formula:limitProcesses}
    \left( t^{-q} \int_0^t r^{q} \mathrm{d}B_s\right)_{t \in \mathbb{R}^+} \quad \quad \text{ for } q \in (-1/2 ,1]
\end{equation}
where $B=(B_t)_{t  \geq 0 }$ stands for a standard Brownian motion, and the bijection only takes place if we exclude $q=0$. The limits of the positive step-reinforced random walks correspond to stochastic integrals (\ref{formula:limitProcesses}) for $q\in (-1/2,0]$ while for negative step-reinforced random walks, $q \in [0,1]$. The case $q=0$, i.e. when the limit is just a Brownian motion, corresponds to both $p=1/2$ for the positive step reinforced case and $p=1$ for the negative step reinforced case, the latter being just Donsker's invariance principle. \medskip \\
}
\end{toexclude}

\section{Introduction}
In short, the purpose of this work is to establish invariance principles for \textit{random walks with step reinforcement},  a particular class of random walks with memory that has been of increasing interest in recent years. Historically,  the so-called \textit{elephant random
walk} (ERW) has been an  important and fundamental example of a step-reinforced random walk that was originally introduced in the physics literature by Schütz and Trimper \cite{SchuetzTrimper} more than  15 years ago. We shall first recall the setting of the ERW in order to motivate the two types of reinforcement that we will work with.

The ERW is a one-dimensional discrete-time nearest neighbour random walk with infinite memory, in allusion to the traditional saying that \textit{an elephant never forgets}. It can be depicted as follows: Fix some $q \in (0,1)$, commonly referred to as the \textit{memory parameter}, and suppose that an elephant makes an initial step in $\{-1,1\}$ at time $1$. After, at each time $n \geq 2$, the elephant selects uniformly at random a step from its past; with probability $q$, the elephant repeats the remembered step, whereas with complementary probability $1-q$ it makes a step in the opposite direction. Notably, in the  case $q= 1/2$, the elephant merely follows the path of a simple symmetric random walk. Notably, the ERW is a time-inhomogeneous Markov chain (although some works in the literature improperly assert its non-Markovian character). The ERW has generated a lot of interest in recent years, a non-exhaustive list of references (with further references therein) is  \cite{BaurBertoin}, \cite{Bercu}, \cite{BercuLaulinCenter}, \cite{Bertenghi}, \cite{GavaSchuetzColetti}, \cite{Coletti2017}, \cite{Coletti2019}, \cite{GuevaraERW},  \cite{KubotaTakei}, \cite{Kuersten}, see also \cite{BaurClass}, \cite{GonzalesMult} for variations.
A striking feature that has been pointed at in those works, is that the long-time behaviour of the ERW exhibits a phase transition at some critical memory parameter. Functional limit theorems for the ERW were already proved by Baur and Bertoin in \cite{BaurBertoin}  by means of limit theorems for  random urns. Indeed, the key observation is that the dynamics of the ERW can be expressed in terms of Pólya-type urn experiments and fall in the framework of the work of Janson \cite{JansonLimitBranchingProcesses}. For a strong invariance principle  for the ERW, we refer to Coletti, Gava and Schütz in \cite{Coletti2017}. 

The  framework of the ERW is however limited to the case of Rademacher distributed steps, and it is natural to look for generalisation of its dynamics that allow the typical step to have an arbitrary distribution  on $\mathbb{R}$. In this work, we aim to study the more general framework of \textit{step-reinforced random walks}.  We shall discuss two such generalisations, called positive and negative step-reinforced random walks, the former generalising the ERW when  $q \in (1/2, 1)$ while the later covers the spectrum $q \in [0,1/2]$, in both cases when the typical step is Rademacher distributed. We start by introducing the former. For the rest of the work, $X$ stands for a random variable  that we assume belongs to $L^2(\mathbb{P})$, we denote by $\sigma^2$ its variance and by $\mu$ its law. Moreover, unless specified otherwise,  $(S_n)$ will always denote a random walk with typical step distributed as $\mu$.  \par 
\textbf{The noise reinforced random walk:} A\textit{ (positive) step-reinforced random walk} or \textit{noise reinforced random walk} is a generalisation of the ERW, where the distribution of a typical step of the walk is allowed to have an arbitrary distribution on $\mathbb{R}$,  rather than just Rademacher. The impact of the reinforcement is still described in terms of a fixed parameter $p \in (0,1)$, that we also refer to as the \textit{memory parameter} or the \textit{reinforcement parameter}. We will work with different values of $p$ but for readability purposes  $p$ does not explicitly appear in the notation or terminology used in this work.
\par   Vaguely speaking, the dynamics are as follows: at each discrete time, with probability $p$ a step reinforced random walk repeats one of its preceding steps chosen uniformly at random, and otherwise, with complementary probability $1-p$, it has an independent increment with a fixed but arbitrary distribution. More precisely, given an underlying probability space $( \Omega, \mathcal{F}, \PP)$ and a sequence ${X}_1, {X}_2, \dots $ of i.i.d. copies of the random variable ${X}$  with law $\mu$,  we define ${\hat{X}}_1, {\hat{X}}_2, \dots$ recursively as follows: First, let $( \varepsilon_i : i \geq 2)$ be an independent sequence of Bernoulli random variables with parameter $p \in (0,1)$ and also consider  $(U[i] : i \geq 2)$ an independent sequence where each $U[i]$ is  uniformly distributed on $\{1, \dots , i \}$.   We set first ${\hat{X}}_1= {X}_1$, and next for $i \geq 2$, we let 
\begin{align*}
{\hat{X}}_i = 
\begin{cases} {X}_i, & \text{if } \varepsilon_i=0, \\
\hat{X}_{U[i-1]} , & \text{if } \varepsilon_i=1. 
\end{cases}
\end{align*}
 Finally, the sequence of the partial sums 
\begin{align*}
{\hat{S}}_n:= {\hat{X}}_1 + \dots + {\hat{X}}_n, \quad n \in \mathbb{N},
\end{align*}
is referred to as a \textit{positive step-reinforced random walk}. From the algorithm, we have for any bounded measurable $f: \mathbb{R}\mapsto \mathbb{R}^+$, 
\begin{equation*}
    \EE(f(\hat{X}_{n+1})) =  (1-p) \EE(f({X}_{n+1})) + \frac{p}{n} \sum_{j=1}^n \EE(f(\hat{X}_j))  
\end{equation*}
and it follows by induction that each $\hat{X}_n$  has law  $\mu$. Beware however that the sequence $(\hat{X}_i)$ is not stationary. Notice that if $(\hat{S}_n)$ is not centred, it is often fruitful to reduce our analysis to the centred case by considering $(\hat{S}_n - n \EE(X))$, which is a centred noise reinforced random walk with typical step distributed as $X - \EE(X)$. Observe that in the degenerate case $p=1$, the dynamics of the positive step-reinforced random walk become essentially deterministic. Indeed when $p=1$ we have $\hat{S}_n=nX_1$ for all $n \geq 1$, in particular the only remaining randomness for this process stems from the random variable $X_1$.

In this setting, when $\mu$ is the Rademacher distribution, Kürsten \cite{Kuersten} (see also \cite{Gonzales}) pointed out that $\hat{S}= ( \hat{S}_n)_{n \geq 1}$ is a version of the elephant random walk with memory parameter $q=(p+1)/2 \in (1/2,1)$ in the present notation. The remaining range of the memory parameter can be obtained by a simple modification that we will address when we introduce random walks with \textit{negatively reinforced steps}. When $\mu$ has a symmetric stable distribution, $\hat{S}$
is the so-called \textit{shark random swim} which has been studied in depth by Businger \cite{Businger}. More general versions when the distribution $\mu$ is infinitely divisible have been considered by Bertoin in \cite{BertoinNoise}, and we will briefly comment on this setting in a moment. Finally,  when we replace the sequence of Bernoulli  random variables $(\epsilon_n)$ by a deterministic sequence $(r_n)$ with $r_n \in \{ 0,1 \}$, the scaling exponents of the corresponding step  reinforced random walks have been  studied by Bertoin in \cite{BertoinScalingExponents}.

In stark contrast to the ERW, the literature available on general step-reinforced random walks remains rather sparse. Quite recently, Bertoin \cite{BertoinUniversality} established an invariance principle for the step-reinforced random walk in the diffusive regime $p \in (0,1/2)$. Bertoin's work concerned a rather simple real-valued and centered Gaussian process $\hat{B}=( \hat{B}(t))_{t \geq 0}$ with covariance function given by
\begin{align} \label{Processes:NoiseReinforcedBM}
\EE \left( \hat{B}(t) \hat{B}(s) \right) = \frac{t^ps^{1-p}}{1-2p} \quad \text{for } 0 \leq s \leq t \quad \text{and } p \in (0,1/2).
\end{align}
This process has notably appeared as the scaling limit for diffusive regimes of the ERW and other  Polya urn related processes, see \cite{BaurBertoin, GavaSchuetzColetti}, \cite{Bertenghi} for higher dimensional generalisations, and \cite{GaussianAproximation_BaiHuZhang}. In \cite{BertoinUniversality} the process displayed in (\ref{Processes:NoiseReinforcedBM}) is referred to as a \textit{noise reinforced Brownian motion} and   belongs to a larger class of reinforced processes recently introduced by Bertoin in \cite{BertoinNoise} called \textit{noise reinforced Lévy processes}. The noise reinforced Brownian motion  plays, in the framework of noise reinforced Lévy processes,  the same role as  the standard Brownian motion in the context of Lévy processes. Moreover, just as the standard Brownian motion $B$ corresponds to the integral of a white noise, $\hat{B}$ can be thought of as the integral of a reinforced version of the white noise, hence the name. More precisely, from (\ref{Processes:NoiseReinforcedBM}) it readily follows that the law of $\hat{B}$ admits the following integral representation 
\begin{align*}
    \hat{B}(t)=t^p \int_0^t s^{-p} \mathrm{d}\beta^r_s, \quad t \geq 0,
\end{align*}
where $\beta^r=(\beta^r_s)_{ s \geq 0}$ is a standard Brownian motion, or equivalently, $\hat{B}=( \hat{B}(t))_{t \geq 0}$ has the same law as
\begin{align*}
    \left( \frac{t^p}{\sqrt{1-2p}}B(t^{1-2p})  \right)_{t \geq 0}. 
\end{align*}
Some further  properties of the noise reinforced Brownian motion can be found in  \cite{BertoinUniversality}, where the following functional limit theorem  \cite[Theorem 3.3]{BertoinUniversality} has been established: let  $p \in (0,1/2)$  and suppose that $X \in L^2(\PP)$. Then, we have the weak convergence of the scaled sequence in the sense of Skorokhod as $n$ tends to infinity
\begin{equation} \label{theorem:InvarianceDiffusiveRegime}
   \left(  \frac{\hat{S} (\lfloor nt \rfloor)-nt \EE(X)}{\sqrt{\sigma^2 n}}  \right)_{t \in \mathbb{R}^+} \Longrightarrow ( \hat{B}_t )_{t \in \mathbb{R}^+}
\end{equation}
where $(\hat{B}_t)_{t \geq 0}$ is a noise reinforced Brownian motion.\par 
Our work generalises this result but our approach differs from \cite{BertoinUniversality} as we work with a discrete martingale introduced by Bercu \cite{Bercu} for the ERW and later generalised in \cite{BertenghiAsymptotic} for step-reinforced random walks. The martingale we work with is a discrete-time stochastic process of the form $\bh{a}_n \hat{S}_n$, where $(\bh{a}_n)_{n \geq 0}$ is a properly defined sequence of positive real numbers of order $n^{-p}$. As we shall see, investigation of said martingale and in particular its quadratic variation process,  in conjunction with the functional martingale CLT \cite{WhittFCLT}, yields an alternative proof of Theorem 3.3 in \cite{BertoinUniversality}.

\textbf{The counterbalanced random walk:} Next we turn our attention to the second process of interest, called the \textit{counterbalanced random walk} or \textit{negative step-reinforced random walk}, introduced recently by Bertoin in \cite{BertoinCounterbalancing}. Beware that  $p$ in our work  always corresponds to the probability of a repetition event, while in \cite{BertoinCounterbalancing} this happens with probability $1-p$.  Similarly, we will consider a sequence of i.i.d. random variables $(X_n)_{n \in \mathbb{N}}$ with distribution $\mu$ on $\mathbb{R}$ and at each time step,  the step performed by the walker will be, with probability $1-p \in (0,1)$,  an independent step $X_n$ from the previous ones  while with complementary probability $p$, the new step is one of the previously performed steps,  chosen uniformly at random, with its sign changed. This last action will be referred to as a \textit{counterbalance} of the uniformly chosen step.  In particular, when $\mu$ is the Rademacher distribution, we obtain an ERW with parameter $(1-p)/2 \in [0,1/2]$.\par 
Formally,  recall that  ${X}_1, {X}_2, \dots $ is a sequence of i.i.d. copies of  ${X}$ and $( \varepsilon_i : i \geq 2)$ is an independent sequence of Bernoulli random variables with parameter $p \in (0,1)$.  We define the sequence of increments ${\check{X}}_1, {\check{X}}_2, \dots$ recursively as follows (beware of the difference of notation between $\hat{X}$ and $\check{X}$): we set first ${\check{X}}_1= {X}_1$, and next for $i \geq 2$, we let 
\begin{align*}
{\check{X}}_i = \begin{cases} {X}_i, & \text{if } \varepsilon_i=0, \\
-\check{X}_{U[i-1]} & \text{if } \varepsilon_i=1 \end{cases},
\end{align*}
where $U[i-1]$ denotes an independent uniform random variable in $\{ 1, \dots , i-1\}$. Finally, the sequence of partial sums 
\begin{align*}
{\check{S}}_n:= {\check{X}}_1 + \dots + {\check{X}}_n, \quad n \in \mathbb{N},
\end{align*}
is referred to as a \textit{counterbalanced random walk (or random walk with negatively reinforced steps)}. Notice also that, in contrast with the positive step-reinforced random walk, when $p=1$ we still get a stochastic process, consisting of consecutive counterbalancing of the initial step $X_1$ while for $p=0$ we just get the dynamics of a random walk. 
 For the positive reinforced random walk we already pointed out that the steps are identically distributed and hence are centred as soon as $X$ is centred. On the other hand, for the negatively step-reinforced case  the recurrent equation on page 3 of \cite{BertoinCounterbalancing} 
\begin{equation*}
    \EE \left( {\check{S}_{n+1}} \right) = (1-p) m + (1-p/n) \EE \left({\check{S}_n}\right),   \quad \quad n \geq 1
\end{equation*}
with initial condition $\EE ({\check{S}_1})=\EE(\check{X}_1) = m$, yields that the process $(\check{S}_n)$ is also centered if $X$ is centred. Observe  however that in stark contrast to the positive step-reinforced random walk, we cannot say that the typical step is centered without loss of generality: Indeed, since $n \mapsto \EE(\check{X}_n)$ is no longer constant as soon as $m \neq 0$, due to the random swap of signs in the negative reinforcement algorithm, the centered process $(\check{S}_n-\EE(\check{S}_n))$ is also no longer a counterbalanced random walk.\par
Turning our attention to its asymptotic behaviour, Proposition 1.1 in  \cite{BertoinCounterbalancing} shows that the behaviour of the counterbalanced random walk $\check{S}_n$ is ballistic. More precisely, denoting by $m=\EE(X)$  the  mean of the typical step $X$, then for all $p \in [0,1]$ the process $(\check{S}_n)$ satisfies a law of large numbers:
\begin{equation*}
    \lim_{n \rightarrow \infty} \frac{\check{S}_n}{n} = \frac{(1-p)m}{1+p} \quad \text{in probability.}
\end{equation*}
Moreover, Theorem 1.2 in \cite{BertoinCounterbalancing} shows that if we also assume that the second moment $m_2 = \EE(X^2)$ is finite, then the fluctuations are Gaussian for all choices $p \in [0,1)$: 
\begin{equation*}
    \frac{\check{S}_n - \frac{1-p}{1+p}mn}{\sqrt{n}} \implies \mathcal{N} \left( 0 , \frac{m_2- \left( \frac{1-p}{1+p}m \right)^2}{1+2p} \right).
\end{equation*}
In particular, when $X$ is centred as will be our case, we simply get 
\begin{equation*}
    \frac{\check{S}_n}{\sqrt{\sigma^2  n}} \implies \mathcal{N} \left( 0 , {(1+2p)^{-1}} \right).
\end{equation*}
On the other hand, when $p =1$ which corresponds to the purely counterbalanced case, and under the additional assumption that $X$ follows the Rademacher distribution, then 
\begin{equation*}
    \frac{1}{\sqrt{n}} \check{S}_n \Longrightarrow \mathcal{N}(0,1/3). 
\end{equation*}
The proofs of these results rely on remarkable connections with random recursive trees and even if these will  not be needed  in the present work, we encourage the interested reader to consult \cite{BertoinCounterbalancing} for more details. In this article, we will establish a functional version of the asymptotic normality mentioned above under the additional assumption that $m=0$, i.e. the typical step is centered. We recall that this assumption cannot be made without the loss of generality. \par
In the same spirit as in the noise-reinforced setting, we will call a \textit{noise counterbalanced Brownian motion of parameter}  $p \in [0,1)$ a Gaussian process $\check{B}$  with covariance given by 
\begin{align} \label{Processes:CounterbalancedBM}
\EE \left( \check{B}(t) \check{B}(s) \right) = \frac{1}{2p+1} \frac{s^{1+p}}{t^p} \quad \text{for } 0 \leq s \leq t \quad \text{and } p \in (0,1), 
\end{align}
and it follows that the law of $\check{B}$ admits the following integral representation 
\begin{equation} \label{definition:counterbalancedBrownianMotion}
    \check{B}_t  = t^{-p} \int_0^t s^{p} \mathrm{d}\beta^c _s, \quad \quad t \geq 0
\end{equation}
in terms of a standard Brownian motion $\beta^c$. 
\par   
\textbf{The invariance principles: }Before stating the functional versions of the results we just mentioned, notice  that given a sample of i.i.d. random variables $(X_n)$ with law $\mu$, and an additional independent collection $(\epsilon_i)$, $(U[i])$ of Bernoulli random variables and uniform random variables respectively as before, we can  construct from the same sample simultaneously to the associated random walk $(S_n)$, the processes  $(\hat{S}_n)$ and $(\check{S}_n)$, that we refer respectively as  \textit{the} positive step-reinforced version and \textit{the} negative step-reinforced version of $({S}_n)$. It is then natural to compare the dynamics of the triplet  $(S_n, \hat{S}_n, \check{S}_n)$, instead of individually working with $(\hat{S}_n)$ and $(\check{S}_n)$. When considering such a triplet, it will always be implicitly assumed that $(\hat{S}_n), (\check{S}_n)$ have been constructed in this particular way from $(S_n)$.  In particular, we used the same sequence of uniform and Bernoulli random variables to define both reinforced versions.  Now we have all the ingredients to state our first main result: 
\begin{thm}\label{thm:ConvergenceTripletDifusive} Fix $p \in [0,1/2)$ and consider the triplet $(S_n, \hat{S}_n, \check{S}_n)$ consisting of the  random walk $(S_n)$ with its reinforced version and its counterbalanced version of parameter $p$.  Assume further that $X$ is centred. Then, the following weak convergence holds in the sense of Skorokhod as $n$ tends to infinity,
\begin{equation} \label{equation:ConvergenceTripletDifusive}
    \left( \frac{1}{\sigma \sqrt{n}} {S}_{\lfloor nt \rfloor}, \frac{1}{\sigma \sqrt{n}} \hat{S}_{\lfloor nt \rfloor}  , \frac{1}{\sigma \sqrt{n}} \check{S}_{\lfloor nt \rfloor}  \right)_{t \in \mathbb{R}^+} 
    \implies   
    \left( B_t , \hat{B}_t , \check{B}_t \right)_{t \in \mathbb{R}^+}
\end{equation}
where $B$, $\hat{B}$, $\check{B}$ denote respectively a standard BM, a noise reinforced BM and a counterbalanced BM with covariances, $\EE(B_s \check{B}_t) = t^{-p}(t \wedge s)^{p+1}(1-p)/(1+p)$, $\EE( B_s \hat{B}_t) = t^p (t \wedge s)^{1-p}$, $\EE( \hat{B}_t \check{B}_s ) = t^{p}s^{-p} (t \wedge s) (1-p)/(1+p)$.
\end{thm}
Notice that in the case $p=0$, i.e. when no reinforcement events occur, this is just Donsker's invariance principle since $(\check{S}_n), (\hat{S}_n)$ are just the random walk $(S_n)$ and $\hat{B}$, $\check{B}$ are just $B$ and hence, from now on we will assume that $p >0$. The process in the limit admits the following simple integral representation in terms of stochastic integrals 
\begin{equation} \label{equation:integralRepTriplet}
    \left( B_t , \,  t^p \int_0^t s^{-p} \mathrm{d}\beta^r_s , \,  t^{-p} \int_0^t s^{p} \mathrm{d}\beta^c_s \right)_{t \in \mathbb{R}^+}
\end{equation}
where $B=(B_t)_{t \geq 0}$,  $\beta^r = (\beta_t)_{t \geq 0}$, $\beta^c = (\beta_c)_{t \geq 0}$ denote three standard Brownian motions with  covariance structure  $ \EE(  B_s \beta_t^r)  = (1-p) (t \wedge s)$, $\EE(  B_s \beta_t^c)  = (1-p)(t \wedge s)$, $\EE(  \beta_s^r \beta_t^c)  = (t \wedge s)(1-p)/(1+p)$.
\par The restriction on the parameter $p \in (0,1/2)$ comes from the fact that, as we will see, for the noise reinforced random walk only for such parameter the functional version works with this scaling, while the centred hypothesis is a restriction coming from the counterbalanced random walk.   Now we point at some variants with less restrictive hypothesis, holding as long as we no longer consider the triplet. This allows us to drop some of the conditions we just mentioned, and the proofs will be embedded in the proof of Theorem \ref{thm:ConvergenceTripletDifusive}. We start by removing the centred hypothesis when only working with the the pair $(S_n, \hat{S}_n)$ in the diffusive regime $p \in [0,1/2)$.   
\begin{thm}
\label{theorem:InvarianceDiffusiveRegime}
 Let  $p \in [0,1/2)$  and suppose that $X \in L^2(\PP)$. Let  $(S_n)$ be a random walk with typical step distributed as $X$ and denote by $(\hat{S}_n)$ its positive step  reinforced version. Then, we have  weak joint convergence of the scaled sequence in the sense of Skorokhod as $n$ tends to infinity towards a Gaussian process 
\begin{equation}  \label{theorem:jointConvergence1}
   \left(  \frac{{S} (\lfloor nt \rfloor)-nt \EE(X)}{\sigma \sqrt{ n}} ,  \frac{\hat{S} (\lfloor nt \rfloor)-nt \EE(X)}{\sigma \sqrt{ n}}\right)_{t \in \mathbb{R}^+} \Longrightarrow (B_t, \hat{B}_t)_{t \in \mathbb{R}^+}
\end{equation}
where $B$ is a Brownian motion, $\hat{B}$ is a noise reinforced Brownian motion with  covariance   $\EE[B_s \hat{B}_t] = t^p (t \wedge s)^{1-p}$. 
\end{thm}
It follows that the limit process in (\ref{theorem:jointConvergence1}) admits the integral representation 
\begin{equation*}
    \left(  B_t ,  t^p \int_0^t s^{-p} \mathrm{d}\beta_s^r  \right)_{t \in \mathbb{R}^+}
\end{equation*}
where $B=(B_t)_{t \geq 0}$ and $\beta^r  = (\beta^r_t)_{t \geq 0}$ denote two standard Brownian motions with $ \EE ( B_t \beta_s^r)  = (1-p) (t\wedge s)$.  This result extends Theorem 3.3 in \cite{BertoinUniversality}  to the pair $(S , \hat{S})$. Notice that the factor $1-p$ in the correlation can be interpreted in terms of the definition of the noise reinforced random walk, since at each discrete time step,  with probability $1-p$ the processes $\hat{S}$ and $S$ share the same step $X_n$. 

Turning our attention to the counterbalanced random walk, when only working with the pair $(S_n, \check{S}_n)$  we can extend the convergence  to $p \in [0,1)$, and is the content of the following result: 
\begin{thm} \label{theorem:InvarianceCounterBalancedCase}
Let $p\in [0,1)$ and suppose that $X \in L^2(\PP)$ is centred. If $(S_n)$ is a random walk with typical step distributed as $X$ and $(\hat{S}_n)$ is its counterbalanced version of parameter $p$, then we have the weak convergence of the sequence of processes in the sense of Skorokhod as $n$ tends to infinity
\begin{equation} \label{equation:InvarianceCounterbalance}
    \left( \frac{1}{\sigma \sqrt{n}} {S}_{\lfloor nt \rfloor},  \frac{1}{\sigma \sqrt{n}} \check{S}_{\lfloor nt \rfloor}  \right)_{t \in \mathbb{R}^+} \Longrightarrow \left(  B_t , \check{B}_t  \right)_{t \in \mathbb{R}^+} 
\end{equation}
where $B$ is a Brownian motion and $\check{B}$ is a noise counterbalanced Brownian motion with  covariance $\EE[B_s \check{B}_t] = t^{-p}(t \wedge s)^{p+1}(1-p)/(1+p)$ and $\sigma^2 = \EE[X^2]$. {If $p=1$ and $X$ follows the Rademacher distribution, the result still holds and in particular $B$ and $\check{B}$ are independent.}
\end{thm}
Moreover, the limit process in (\ref{equation:InvarianceCounterbalance}) admits the simple integral representation 
\begin{equation*}
    \left( B_t ,  t^{-p} \int_0^t s^{p} \mathrm{d}\beta_s^c  \right)_{t \in \mathbb{R}^+}
\end{equation*}
where $B=(B_t)_{t \geq 0}$ and $\beta^c = (\beta^c_t)_{t \geq 0}$ denote two standard Brownian motions with $ E(  B_s \beta^c_t)  = (1-p) (t \wedge s)$.\par
 Finally, we turn back our attention  to the noise reinforced setting when the parameter is $p = 1/2$. Our method allows us to establish an invariance principle for the step-reinforced random walk at criticality $p=1/2$ but notice that in this case we do not establish a joint convergence, as the required scalings are no longer compatible. 
\begin{thm} \label{theorem:InvarianceCriticalRegime}
Let $p=1/2$ and suppose that $X \in L^2(\PP)$. Then, we have the weak convergence of the sequence of processes in the sense of Skorokhod as $n$ tends to infinity
\begin{equation} \label{equation:InvarianceCriticalRegime}
    \left( \frac{\hat{S}_{\lfloor n^t \rfloor}- n^t  \EE(X)}{\sigma \sqrt{  \log(n)}   n^{t/2}}  \right)_{t \in \mathbb{R}^+} \Longrightarrow (B_t)_{t \in \mathbb{R}^+}
\end{equation}
where $B=(B_t)_{t \geq 0}$ denotes a standard Brownian motion.
\end{thm}
Our proofs rely on a version of the martingale Functional Central Limit Theorem (abreviated MFCLT), which we  state  for the reader's convenience.  For more general versions, we refer to Chapter VIII in  \cite{Processes}. If $M=(M^{1}, \dots , M^{d} )$ is an rcll $d$-dimentional process, we denote by $\Delta M$ its jump process, which is the d-dimensional process null at 0 defined as $( M^{1}_t- M^{1}_{t-}, \dots , M^{d}_t- M^{d}_{t-})_{t \in \mathbb{R}^+}$.
\begin{thm}[MFCLT, VIII-3.11 from \cite{Processes}] \label{thm:FCLT}
Assume $M= (M^1, \dots , M^d)$ is $d$-dimentional continuous Gaussian martingale with independent increments,  and predicable covariance process 
    $(\langle M^i,M^j \rangle)_{i,j \in \{ 1,\dots , d \}}$. For each $n$, let $M^n = (M^{n,1}, \dots , M^{n,d})$ be a $d$-dimentional local martingale with uniformly bounded jumps $\|\Delta M^n\| \leq K$ for some constant $K$. The following conditions are equivalent: 
\begin{enumerate}[(i)]
    \item $M^n \Rightarrow M$ in the sense of Skorokhod,
    \item There exists some dense set $D\subset \mathbb{R}^+$ such that for each $t \in D$ and $i,j \in \{1,\dots , d \}$, \\as $n \uparrow \infty$, 
\begin{equation} \label{thm:mfclt-quadraticVariation}
    \langle M^{n,i} , M^{n,j} \rangle_t \rightarrow  \langle M^{i} , M^{j} \rangle_t \quad \quad  \text{ in probability,}
\end{equation}
and 
\begin{equation} \label{thm:mfclt-negligeableJumps}
    \sup_{s \leq t} |\Delta M_s^n| \rightarrow 0  \hspace{25mm}  \text{ in probability}.
\end{equation}
\end{enumerate}
\end{thm}
\noindent$ \blacklozenge$ \textit{The rest of this paper is organised as follows:}
In Section \ref{section:martingale} we introduce a crucial martingale for our reasoning associated with step-reinforced random walks and investigate its properties. We derive maximal inequalities and asymptotic results for the noise reinforced random walk that will be needed in the sequel. 
Section \ref{section:proofBoundedSteps} is devoted to the proof of Theorem \ref{thm:ConvergenceTripletDifusive} under the additional assumption that the typical step $X$ is bounded and in section \ref{section:reductionArgument} we discuss how to relax this assumption to the general case of unbounded steps by a truncation argument. In the process, we will also deduce the proofs of Theorem \ref{theorem:InvarianceDiffusiveRegime} and Theorem \ref{theorem:InvarianceCounterBalancedCase}. Finally, in Section \ref{section:criticalRegime} we address the proof of Theorem \ref{theorem:InvarianceCriticalRegime} and we shall again proceed in two stages. Since many arguments can be carried over from the previous sections, some details are skipped.

\section{The martingales associated to a reinforced random walk}\label{section:martingale}
\textit{In this section we work under the additional assumption that the typical step $X \in L^2(\PP)$ is centred and recall that we  denote by $\sigma^2= \EE(X^2)$ its variance. The centred hypothesis is maintained for Sections 3 and 4, but dropped in Section 5.
} \smallskip \\

Recall that if $M = (M_n)_{n \geq 0}$ is a discrete-time real-valued and square integrable  martingale, then its predicable variation process $\langle M \rangle $ is the process defined by $\langle M \rangle_0=0$ and for $n \geq 1$, 
\begin{equation*}
    \langle M \rangle_n = \sum_{k=1}^n \EE ( \Delta M_k^2 \mid  \mathcal{F}_{k-1}),
\end{equation*}
while if $(Z_n)$ is another martingale,  the predictable covariation of the pair $\langle M , Z \rangle$  is the process defined by $\langle M , Z \rangle_0 = 0$ and for $n \geq 1$, 
\begin{equation*}
    \langle M,  Z \rangle_n =  \sum_{k=1}^n \EE ( \Delta M_k \Delta Z_k \mid \mathcal{F}_{k-1} ). 
\end{equation*}
We define two sequence $(\bh{a}_n, n \geq 1)$, $(\bc{a}_n, n \geq 1)$  as follows:  Let $\bh{a}_1=\bc{a}_1=1$ and for  each $n \in \{2, 3,  \dots \}$, set 
\begin{equation} \label{forumla:multiplicativeTerm}
   \bh{a}_n = \prod_{k=1}^{n-1} \bh{\gamma}_k^{-1} =\frac{\Gamma(n)}{\Gamma(n+p)}, 
    \hspace{20mm}
    \bc{a}_n = \prod_{k=1}^{n-1} \bc{\gamma}_k^{-1} =\frac{\Gamma(n)}{\Gamma(n-p)}
\end{equation}
for respectivelly $\bh{\gamma}_n = \frac{n+p}{n} $, $\bc{\gamma}_n = \frac{n-p}{n}$ when $n \geq 2$. 
\begin{prop} \label{prop:MultMartStart} 
The processes $\hat{M} = (\hat{M}_n)_{n \geq 0}$, $\check{M} = (\check{M}_n)_{n \geq 0}$
defined as $\hat{M}_0 = \check{M}_0 = 0$ and $\hat{M}_n = \bh{a}_n \hat{S}_n$, $\check{M}_n = \bc{a}_n \check{S}_n$ for $n \geq 1$ are centred square integrable martingales and we denote the natural filtration generated by the pair by  $(\mathcal{F}_n)$, where $\mathcal{F}_0$ is the trivial sigma-field. Further, their respective predictable quadratic variation processes   is given by $\langle \hat{M} \rangle_0= \langle \check{M} \rangle_0 = 0 $ and, for all $n \geq 1$
\begin{align}
\label{formula:TheQuadraticVariation}
    \langle \hat{M} \rangle_n = \sigma^2 + \sum_{k=2}^n  \bh{a}_k^2 \left( (1-p)\sigma^2 - p^2 \left( \frac{\hat{S}_{k-1} }{k-1} \right)^2 + p \frac{\hat{V}_{k-1}}{k-1} \right) 
\end{align}
and 
\begin{align}
\label{formula:TheQuadraticVariationBalanced}
    \langle \check{M} \rangle_n = \sigma^2 + \sum_{k=2}^n \bc{a}_k^2 \left( (1-p)\sigma^2 - p^2 \left( \frac{\check{S}_{k-1} }{k-1} \right)^2 + p \frac{\hat{V}_{k-1}}{k-1} \right) 
\end{align}
where $(\hat{V}_n)_{n \geq 1}$ is the step-reinforced process given by $\hat{V}_n= \hat{X}_1^2 + \dots + \hat{X}_n^2$ and the sums should be considered identical to zero for $n=1$.  
\end{prop}
\begin{proof}
Starting with the positive-reinforced case, notice that for any $n \geq 1$ we have 
\begin{align} \label{formula:conditionalRstep}
\EE \left( \hat{X}_{n+1} \mid \mathcal{F}_n \right) = (1-p) \EE (X) + p \frac{\hat{X}_1 + \dots + \hat{X}_n}{n} = p \frac{\hat{S}_n}{n}.
\end{align}
Hence, since $\hat{S}_{n+1}=\hat{S}_n+\hat{X}_{n+1}$, and $ \bh{\gamma}_n = (n+p)/n$, 
\begin{align}\label{proofstep:NeededForMartingale}
\EE( \hat{S}_{n+1} \mid \mathcal{F}_n) =  \bh{\gamma}_n \cdot \hat{S}_n
\end{align} 
and therefore, we obtain
\begin{align*}
\EE(\hat{M}_{n+1} \mid \mathcal{F}_n)= \bh{a}_{n+1} \EE(  \hat{S}_{n+1} \mid \mathcal{F}_n) = \bh{a}_{n+1} \bh{\gamma}_n \check{S}_n =  \bh{a}_n {\hat{S}_n}=\hat{M}_n.
\end{align*}
Moreover, as $X$ is centred and the steps $(\hat{X}_k)$ are identically distributed by what was discussed in the introduction, we have 
\begin{align*}
\EE( \hat{M}_n) = \EE( \hat{X}_1)=\EE(X)=0
\end{align*}
and we conclude that $(\hat{M}_n)_{n \geq 0}$ is a  martingale. Turning our attention to its quadratic variation, we have $\EE( \hat{S}_n^2) \leq n^2 \EE(X^2)= n^2 \sigma^2$ and hence, $\hat{M}_n$ is indeed square integrable and its predictable quadratic variation exists.
Next, we observe that for $n \geq 1$ we have
\begin{align}
\EE(\hat{M}_{n+1}^2-\hat{M}_n^2 \mid \mathcal{F}_n) &= \EE( (\hat{M}_{n+1}-\hat{M}_n)^2 \mid \mathcal{F}_n) \notag \\
&= \bh{a}_{n+1}^2 {\EE ( (\hat{X}_{n+1}- \EE(\hat{X}_{n+1} \mid \mathcal{F}_n))^2 \mid \mathcal{F}_n )} \notag \\
&= \bh{a}_{n+1}^2 \left( \EE(\hat{X}_{n+1}^2 \mid \mathcal{F}_n)- ( \EE(\hat{X}_{n+1} \mid \mathcal{F}_n))^2 \right)\notag \\
&= \bh{a}_{n+1}^2 \left( {\EE(\hat{X}_{n+1}^2 \mid \mathcal{F}_n) - \frac{p^2}{n^2}\hat{S}_n^2}\right). \label{eq:RequiredForProof} 
\end{align}
Finally, as was pointed out in the proof of Lemma 3 in \cite{BertoinNoise}, and can be verified from the definition of the $\hat{X}_n$,  it holds that 
\begin{align} \label{eq:SecondMomentStep}
    \EE( \hat{X}_{n+1}^2 \mid \mathcal{F}_n) = p \frac{\hat{V}_n}{n} +(1-p) \sigma^2,
\end{align}
and hence we arrive at the formula (\ref{formula:TheQuadraticVariation}).
\par For the negative-reinforced case, the proof follows very similar steps after minor modifications have been made. Since 
\begin{align}\label{formula:conditionalCstep} 
    \EE (\check{X}_{n+1} \mid \mathcal{F}_n ) 
    = m (1-p) - p \frac{\check{X}_1+ \dots + \check{X}_n}{n}
    =  -\frac{p}{n} \check{S}_n
\end{align}
we now have 
\begin{equation}\label{equation:conditionalStepCounterbalanced}
    \EE (\check{S}_{n+1} \mid \mathcal{F}_n ) = \left( \frac{n-p }{n} \right) \check{S}_n =  \bc{\gamma}_n {\check{S}_n}, 
\end{equation}
and the martingale property for $(\check{M}_n)_{n \geq 0}$ follows. For the quadratic variation, the proof is the same after noticing that since clearly $\check{X}_k^2 = \hat{X}_k^2$, we  can also write   $\hat{V}_n= \check{X}_1^2 + \dots + \check{X}_n^2$.
\end{proof}
We write for further use the following asymptotic behaviours:  the first ones are related to the study of the positive-reinforced case and hold for $p \in (0,1/2)$:
\begin{equation} \label{equation:asymptSeries} 
\lim_{n \rightarrow \infty} n^{2p-1} \sum_{k=1}^n \bh{a}_k^2 = \frac{1}{1-2p}, \hspace{25mm}  \bh{a}_n = \frac{\Gamma(n) }{\Gamma(n+p)}  \sim n^{-p}  \quad \text{ as } n \uparrow  \infty
\end{equation}
while for $p = 1/2$ we have a change on the asymptotic behaviour in the series,
\begin{equation}
  \label{equation:asymptSeriesCritical}
    \lim_{n \rightarrow \infty} \frac{1}{\log(n)} \sum_{k=1}^n \bh{a}_k^2 = 1,  \hspace{32mm}  \bh{a}_n = \frac{\Gamma(n) }{\Gamma(n+1/2)}  \sim n^{-\frac{1}{2}}  \quad \text{ as } n \uparrow  \infty
\end{equation}
which is the reason behind the different scaling showing in Theorem \ref{theorem:InvarianceCriticalRegime}. On the other hand, for the negatively-reinforced case we have for $p \in (0,1)$, 
\begin{equation} \label{formula:asyptoticSeriesCounterbalanced}
    \lim_{ n \rightarrow \infty } \frac{ 1 }{n^{1+2p}}  \sum_{k=1}^{n} \bc{a}^2_k = \frac{ 1 }{{1+2p}},
    \hspace{25mm}
    \bc{a}_n = \frac{\Gamma(n) }{\Gamma(n-p)} \sim n^{p}  \quad \text{ as } n \uparrow \infty.
\end{equation}
The limits are derived from standard Gamma function asymptotic behaviour, and were already pointed out in Bercu \cite{Bercu}. We first focus our attention on a law of large numbers, that will be needed in the proof of  Theorem \ref{theorem:InvarianceDiffusiveRegime}. 
\begin{lem} \label{thm:LLNDiffusive} Suppose that $\|X \|_\infty < \infty$ and $p \in (0,1/2)$.
We have the almost sure convergence
\begin{equation*}
    \lim_{n \to \infty} \frac{\hat{S}_n}{n^{1-p}} =0 \quad \text{a.s.}
\end{equation*}
and a fortiori, $\lim_{n \to \infty} n^{-1}{\hat{S}_n}=0 \text{ a.s.}$
\end{lem}
\begin{proof}
The proof of Theorem \ref{thm:LLNDiffusive} is adapted from \cite{Bercu} and outlined here for the sake of completeness. Under the standing assumption that the typical step $X$ is bounded,  we gather from (\ref{eq:RequiredForProof}) that 
\begin{align} \label{notation:vn}
    \langle \hat{M} \rangle_n \leq \nu_n:= \|X\|_\infty^2 \sum_{k=1}^n \bh{a}_{k+1}^2.
\end{align}
By (\ref{equation:asymptSeries}) we have as $n \to \infty$
 \begin{equation*}
  \nu_n \sim n^{1-2p}    \cdot  \|X\|_\infty \frac{1}{1-2p} 
 \end{equation*}
more precisely, as $p<1/2$, the sequence $\nu_n$ increases to infinity with power polynomial rate  of $n^{1-2p}$. We then obtain from the strong law of large numbers for martingales, see for instance Theorem 1.3.24 in \cite{duflo}, that 
\begin{align*}
    \lim_{n \to \infty} \frac{\hat{M}_n}{\nu_n}=0 \quad \text{a.s.}
\end{align*}
and our claim follows. 
\end{proof}
We continue by investigating bounds for the second moments of the supremum process of the step-reinforced random walk $\hat{S}$ for all regimes, and then deduce related LLN type results that will also be needed afterwards. 
\begin{lem}  \label{lem:maximalIneq}
For every $n \geq 1$, the following bounds hold for some numerical constant $c$:
\begin{equation*}
    \sigma^{-2} \EE \left( \sup_{k \leq n} |\hat{S}_k|^2 \right)  \leq 
    \begin{cases}
    c  n, & \text{ if } p \in (0, 1/2)\\
    c  n \log n,  & \text{ if } p = 1/2\\
    c  n^{2p}, & \text{ if } p \in  (1/2,1).
    \end{cases}
\end{equation*}
\end{lem}
\begin{proof} We tackle each of the three cases $p \in (0,1/2)$, $p=1/2$ and $p \in (1/2,1)$ individually:
\begin{enumerate}[(i)]
    \item Let us first consider the case when $p \in (0, 1/2)$. We observe that by (\ref{eq:RequiredForProof}) and by (\ref{equation:asymptSeries})
\begin{align*}
    \EE(\hat{M}_n^2) = \EE( \langle \hat{M} \rangle _n)  \leq \sum_{k=1}^{n} \sigma^2 \bh{a}_{k+1}^2 \sim \sigma^2 \frac{1}{1-2p}n^{1-2p}, \quad \text{as } n \to \infty.
\end{align*}
Hence we obtain by Doob's inequality that 
\begin{align*}
    \EE \left( \sup_{k \leq n} | \hat{M}_k|^2 \right) \leq c_1 \sigma^2 n^{1-2p}
\end{align*}
where $c_1>0$ is some constant. Since it evidently holds that
\begin{align*}
    \EE \left( \sup_{k \leq n} | \hat{S}_k|^2 \right) \leq \frac{1}{\bh{a}_n^2}  \EE \left( \sup_{k \leq n} | \hat{M}_k|^2 \right),
\end{align*}
it follows readily that
\begin{align*}
    \EE \left( \sup_{k \leq n} | \hat{S}_k|^2 \right) \leq c_1 \sigma^2  \frac{n^{1-2p}}{\bh{a}_n^2}  \sim c_1 \sigma^2 n, \quad \text{as } n \to \infty. 
\end{align*}
By monotonicity, we conclude the proof for this case.
\item  Let us now assume that $p=1/2$, we then obtain by (\ref{equation:asymptSeriesCritical}) and monotonicity that for all $n \geq 1$ we have 
\begin{align*}
    \EE ( \langle \hat{M} \rangle _n) \leq \sigma^2   \log n.
\end{align*}
We conclude as in the previous case that this implies
\begin{align*}
     \EE \left( \sup_{k \leq n} | \hat{S}_k|^2 \right) \leq c_2 \sigma^2 n \log n,
\end{align*}
where $c_2>0$ is some constant. 
\item  Finally, let us consider the case $p>1/2$. Here, we then have as $n \to \infty$
\begin{align*}
    \sigma^2 \sum_{k=1}^n \bh{a}_{k+1}^2 \leq C\sigma^2 \sum_{k=1}^n \frac{1}{k^{2p}} <  \tilde{c}
\end{align*}
for a constant $C$ large enough and some finite constant  $\tilde{c}$. This entails that $\EE ( \langle \hat{M} \rangle_n) \leq \sigma^2 \tilde{c}$ and we deduce as before the bound 
\begin{align*}
     \EE \left( \sup_{k \leq n} | \hat{S}_k|^2 \right) \leq c_3 \sigma^2 n^{2p},
\end{align*}
where $c_3>0$ is some constant.
\end{enumerate}
Thus we have established the desired bounds.
\end{proof}
As  an application of the maximal inequalities displayed in Lemma \ref{lem:maximalIneq} for the positive reinforced random walk, we establish $L^2$ convergence type results for all regimes $p \in (0,1)$ that will be needed in our proofs: 
\begin{cor} \label{thm:WeakLLN}
We have the following convergences in the $L^2$-sense.
\begin{enumerate}[(i)]
    \item For $p \in (0,1/2)$ we have 
    \begin{equation*}
    \lim_{n \to \infty} \frac{\hat{S}_n}{n^{1-p}} =0.  
\end{equation*}
\item For $p=1/2$ we have
\begin{align*}
    \lim_{n \to \infty} \frac{\hat{S}_n}{\sqrt{n} \log n} =0.
\end{align*}
\item For $p \in (1/2,1)$ we have 
\begin{align*}
    \lim_{n \to \infty} \frac{\hat{S}_n}{n}=0.
\end{align*}
\end{enumerate}
\end{cor}
\begin{proof}
Let $(f(n))$ be a sequence of positive numbers and notice that by Lemma \ref{lem:maximalIneq}, if as $n \uparrow \infty$
\begin{equation*}  
    \begin{cases}
    \frac{1}{f^2(n)}n \rightarrow 0,  &\text{ if } p \in (0,1/2)\\
    \frac{1}{f^2(n)}(n \log n)  \rightarrow 0, & \text{ if } p  = 1/2 \\
    \frac{1}{f^2(n)}n^{2p} \rightarrow 0, & \text{ if } p \in (1/2,1)
    \end{cases},
\end{equation*}
then we have convergence in the $L^2$-sense to 0 of the sequence $(\hat{S}_n/f(n))$. Now respectively for each one of the tree cases:  \begin{enumerate}[(i)]
    \item We take $f(n):= n^{1-p}$ and observe that $n^{2p-1} \to 0$ as $n \to \infty$ since $p \in (0,1/2)$.
    \item We take $f(n):= \sqrt{n} \log n$, plainly $1/ \log(n) \to 0$ as $n \to \infty$.
    \item We take $f(n):=n$ and observe that $n^{2(p-1)} \to 0$ as $n \to \infty$ because $p<1$. 
\end{enumerate}
This concludes the proof.
\end{proof}
We wrap up our discussion by mentioning that in the superdiffusive regime $p \in (1/2,1)$ the convergence displayed in Corollary \ref{thm:WeakLLN} can be improved. The following proposition has already been observed in \cite{BertenghiAsymptotic} using a different technique, we present here a more elementary approach. 
\begin{prop} For every fixed  $p \in (1/2,1)$, we have 
\begin{align*}
  \hspace{30mm}  \lim_{n \to \infty} \frac{\hat{S}_n}{n^p}= \hat{W} \hspace{20mm} \text{ a.s. and in }L^2(\PP),
\end{align*}
where $\hat{W} \in L^2(\PP)$ is a non-degenerate random variable.  
\end{prop}
\begin{proof}
Thanks to Proposition \ref{prop:MultMartStart} we know that $\hat{M}_n = \hat{a}_n \hat{S}_n$ is a martingale. Further, we obtain from (\ref{eq:RequiredForProof}) and the asymptotics $\bh{a}_n \sim n^{-p}$  that, for some constant $C$ large enough, 
\begin{align*}
    \EE( | \hat{M}_n|^2) = \EE ( \langle \hat{M} \rangle_n) \leq \sigma^2 \sum_{k=1}^n \bh{a}_{k+1}^2 \leq C \sum_{k=1}^n \frac{1}{k^{2p}},
\end{align*}
for all $n \in \mathbb{N}$. Since $p>1/2$, the latter series is summable and we conclude that $$\sup_{n \in \mathbb{N}} \EE(| \hat{M}_n|^2) <\infty.$$ By Doob's martingale convergence theorem  there exists  a non-degenerate random variable $\hat{W} \in L^2(\PP)$ such that $\hat{M}_n \to \hat{W}$ a.s. and in $L^2(\PP)$ as $n \to \infty$. Using the asymptotics $\bh{a}_n \sim n^{-p}$ we conclude the proof. 
\end{proof}
\section{Proof of Theorem \ref{thm:ConvergenceTripletDifusive} when $X$ is bounded.} \label{section:proofBoundedSteps}
Recall that in this section and Section \ref{section:reductionArgument} we work under the additional assumption that $X$ is centred. As was discussed in the introduction, for positive step-reinforced random walks the centredness hypothesis can be assumed without loss of generality, but that is no longer the case for negative step-reinforced random walks. We are now in a position to prove Theorem \ref{thm:ConvergenceTripletDifusive} when $X$ is bounded and in the process we will also establish Theorem \ref{theorem:InvarianceDiffusiveRegime}  and Theorem \ref{theorem:InvarianceCounterBalancedCase}. For that reason, in several statements we also consider  $p \in [1/2, 1]$ when working with the counterbalanced random walk. Additionally, when we work with the counterbalanced random walk for $p=1$, we assume as in Theorem \ref{theorem:InvarianceCounterBalancedCase} that $X$ is Rademacher distributed, this will be recalled when necessary. Our approach relies on using the martingale introduced in Proposition \ref{prop:MultMartStart} and applying the  MFCLT \ref{thm:FCLT}. We will establish the general case for $X \in L^2(\PP)$ by a truncation argument, detailed in Section \ref{section:reductionArgument}. 
\par Now, the key is to notice that, since by (\ref{equation:asymptSeries}) resp. (\ref{formula:asyptoticSeriesCounterbalanced}) we have for any $t \geq 0$ 
 \begin{equation*}
     \frac{ \bh{a}_{\lfloor nt \rfloor}}{n^{-p} } \sim t^{-p} \hspace{13mm} \text{ and } \hspace{13mm}  \frac{\bc{a}_{\lfloor nt \rfloor}}{ n^{p}} \sim t^{p} \quad \quad \text{ as } n \uparrow  \infty,
 \end{equation*}
in order to get the convergence (\ref{equation:ConvergenceTripletDifusive}) it is  enough to prove  (except for a technical detail at the origin in the third coordinate that will be properly addressed), the convergence 
 \begin{align}
      & \left( \frac{1}{\sqrt{n}} S_{\lfloor nt \rfloor} , \frac{1}{ \sqrt{n}}  \frac{\bh{a}_{\lfloor nt \rfloor}}{n^{-p}} \hat{S} _{\lfloor nt \rfloor}, \frac{1}{\sqrt{n}} \frac{\bc{a}_{\lfloor nt \rfloor }}{n^{p}} \check{S}_{\lfloor nt \rfloor } \right)_{t \in \mathbb{R}^+} \nonumber \\
      & \hspace{50mm} 
      \Longrightarrow 
      \left( \sigma B_t, \sigma \int_0^t s^{-p} d\beta_s, \sigma \int_0^t s^{p} d\beta_s \right)_{t \in \mathbb{R}^+} \label{equation:martingaleLimit1}
 \end{align}
 for Brownian motions $B$ and $\beta^r$ and $\beta^c$ defined as in (\ref{equation:integralRepTriplet})  where the sequence on the left-hand side is now composed by martingales. More precisely, for each $n \in \mathbb{N}$, the processes  
 \begin{equation} \label{definition:TheMartingalePositiveReinf}
  \left(\hat{N}^{(n)}_t\right)_{t \in \mathbb{R}^+} 
  := \left( \frac{1}{\sqrt{n}} \frac{\bh{a}_{\lfloor nt \rfloor}}{n^{-p}} \hat{S} _{\lfloor nt \rfloor} \right)_{t \in \mathbb{R}^+},
  \hspace{15mm} 
  \left(  \check{N}^{(n)}_t \right)_{t \in \mathbb{R}^+} :=   \left(  \frac{1}{\sqrt{n}} \frac{\bc{a}_{\lfloor nt \rfloor }}{n^{p}} \check{S}_{\lfloor nt \rfloor } \right)_{t \in \mathbb{R}^+} 
 \end{equation}
 are just rescaled, continuous-time versions of the martingales we introduced in Proposition \ref{prop:MultMartStart}, multiplied by respective factors of $n^{p-1/2}$ and $n^{-1-p}$. We will also denote as  $N^{(n)}$  the scaled random walk in the first coordinate and we proceed at establishing (\ref{equation:martingaleLimit1})  by verifying that the conditions of the MFCLT \ref{thm:FCLT} are satisfied. In that direction and recalling the condition (\ref{thm:mfclt-quadraticVariation}),  we start by investigating the asymptotic negligeability of the jumps:
\begin{lem}[Asymptotic negligeability of jumps] \label{lemma:asympnegliblejumps} \ 
\begin{enumerate}[(i)]
    \item  Fix $p \in (0,1/2)$. The following convergence holds in probability: 
    \begin{align*}
    & \sup_{t \geq 0  }  |\Delta \hat{N}^{(n)}_t| \rightarrow 0 \quad \quad \text{ as } n \uparrow \infty.
\end{align*}
    \item  Fix $p \in (0,1)$. For each $T >0$, the following convergence holds in probability 
\begin{equation*}
     \sup_{t \leq T} |\Delta \check{N}_t^{(n)}| \rightarrow 0  \quad \quad \text{ as } n \uparrow \infty. 
\end{equation*}
\end{enumerate}
\end{lem}
\begin{proof}
(i) We will show that we can bound $\sup_t |\Delta \hat{N}^{(n)}_t|$ uniformly by a  function decreasing to 0 as $n \rightarrow \infty$. In that direction, notice that 
 \begin{align*}
     \sup_{t \geq 0  }  |\Delta \hat{N}^{(n)}_t| 
     &= \frac{1}{ \sqrt{n}} \sup_{k \in \mathbb{N} } \left| {n^p} \bh{a}_{k+1} \hat{S}_{k+1} - {n^p} \bh{a}_{k}\hat{S}_{k} \right| \\
     &= \frac{1}{ n^{  1/2-p}} \sup_{k \in \mathbb{N}} \left| \sum_{i=1}^k \hat{X_i} \left( \bh{a}_{k+1} - \bh{a}_{k} \right) + \bh{a}_{k+1} \hat{X}_{k+1} \right| \\ 
     &\leq  \frac{ \| X \|_\infty }{ n^{  1/2-p}} \sup_{k \in \mathbb{N}} \left\{   k   \left| \bh{a}_{k+1} - \bh{a}_{k} \right| + \bh{a}_{k+1}\right\}.
 \end{align*}
 By hypothesis we have $\|X\|_\infty < \infty$ and since  $\bh{a}_{k}$ is decreasing, it is enough to show that
 \begin{equation} \label{equation:boundedSup}
     \sup_{k \in \mathbb{N}}   k   \left(  \bh{a}_{k} - \bh{a}_{k+1} \right) <\infty.
 \end{equation}
 Recalling (\ref{forumla:multiplicativeTerm}) and the asymptotic behaviour (\ref{equation:asymptSeries}), (\ref{equation:boundedSup}) follows from the following estimate for the difference:
 \begin{equation} \label{equation:asymptoticDiferenceHat}
     \bh{a}_{k} - \bh{a}_{k+1} = \bh{a}_{k} \left( 1- \frac{k+1}{k+1 +p} \right) \sim p k^{-(p+1)}.
 \end{equation}
 (ii) Since we also have $\Delta \check{M}_{k+1} = \check{a}_{k+1}(\check{S}_{k+1} - \check{\gamma}_k \check{S}_k)$, we deduce that 
 \begin{align*}
    \sup_{t \leq T} |\Delta \check{N}^{(n)}_t| 
    &= \frac{1}{n^{1/2 + p}} \sup_{k \leq nT} | \bc{a}_{k+1}\left(\check{S}_{k+1} -  \bc{\gamma}_k\check{S}_k \right) | \\
    &= \frac{1}{n^{1/2 + p}} \sup_{k \leq nT} \bc{a}_{k+1} \bigg| \sum_{j=1}^k \check{X}_j (1-\bc{\gamma}_k) +  \check{X}_{k+1}  \bigg| \\
    &\leq \frac{\| X\|_\infty }{n^{1/2 + p}} \sup_{k \leq nT} \bc{a}_{k+1} \left( \sum_{j=1}^k \left( \frac{p}{j} \right) + 1 \right)
\end{align*}
since $\bc{\gamma}_n = (n-p)/n$. Recalling the asymptotic behaviour (\ref{formula:asyptoticSeriesCounterbalanced}) of $\bc{a}_n$, the supremum in the above expression can be uniformly bounded by
\begin{equation*}
      C \cdot (nT)^{p} \log\left( nT \right) 
\end{equation*}
for a constant $C$ large enough, entailing that $\sup_{t \leq T} |\Delta \check{N}^{(n)}_t| \rightarrow 0$ pointwise for each $T$.
\end{proof}
Now we turn our attention to the joint convergence of the quadratic variation process, and this is the content of the following lemma: 
\begin{lem}[Convergence of quadratic variations] \label{lemma:quadvarconv}
For each fixed $t \in \mathbb{R}^+$, the following convergences hold in probability for $p \in (0,1/2)$, unless specified otherwise: 
\begin{enumerate}[(i)]
    \item $\displaystyle \lim_{n \rightarrow \infty}  \langle \hat{N}^{(n)}, \hat{N}^{(n)} \rangle_t = { \sigma^2 } \int_0^t s^{-2p}  \mathrm{d}s.$
    \item $\displaystyle \lim_{n \rightarrow \infty}  \langle \check{N}^{(n)}, \check{N}^{(n)} \rangle_t = { \sigma^2 } \int_0^t s^{2p}  \mathrm{d}s$, $\quad$ for $p \in (0,1]$.
    \item $\displaystyle \lim_{n \rightarrow \infty} \langle \hat{N}^{(n)}, N^{(n)} \rangle_t  = { \sigma^2 } (1-p) \int_0^t s^{-p} \mathrm{d}s.$
    \item $\displaystyle \lim_{n \rightarrow \infty} \langle \check{N}^{(n)}, N^{(n)} \rangle_t  = { \sigma^2 } (1-p) \int_0^t s^{p} \mathrm{d}s$, $\quad$ for $p \in (0,1]$.
    \item $\displaystyle \lim_{n \rightarrow \infty} \langle \hat{N}^{(n)} , \check{N}^{(n)} \rangle_t =  t  \sigma^2 \frac{1-p}{1+p}$.
\end{enumerate}
where for the case $p=1$ in $(ii)$ and $(iv)$ we assume that $X$ is distributed Rademacher. 
\end{lem}
Lemma \ref{lemma:quadvarconv} provides the key asymptotic behaviour for the sequence of  quadratic variations and its proof is rather long.
\begin{proof} We tackle each item \textit{(i)--(v)} individually, item (v) being the most arduous.
\begin{enumerate}[(i)]
    \item For each $n \in \mathbb{N}$, we gather from (\ref{formula:TheQuadraticVariation}) that the predictable quadratic variation of this martingale is given by 
 \begin{align*}
 & \langle \hat{N}^{(n)}, \hat{N}^{(n)} \rangle_t = \\
 & \hspace{5mm} \frac{1}{n^{1-2p}} \left( \sigma^2+ (1-p)\sigma^2 \sum_{k=2}^{\lfloor nt \rfloor} \bh{a}_{k}^2
     - p^2 \sum_{k=2}^{\lfloor nt \rfloor} \bh{a}_{k}^2 \left( \frac{\hat{S}_{k-1} }{k-1} \right)^2 + p \sum_{k=2}^{\lfloor nt \rfloor} \bh{a}_{k}^2 \left(  \frac{{\hat{V}}_{k-1} }{k-1} \right)  \right).
 \end{align*}
 We will study separately the limit as $n \rightarrow \infty$ of the three nontrivial terms, as the first one evidently vanishes. To start with,  it follows readily from (\ref{equation:asymptSeries}) that 
 \begin{equation}
 \label{equation:term1}
     \lim_{n \rightarrow \infty} \frac{{\sigma^2 }}{n^{1-2p}}  (1-p) \sum_{k=2}^{\lfloor nt \rfloor} \bh{a}_{k}^2 = \frac{{\sigma^2 }}{1-2p} t^{1-2p} (1-p).
 \end{equation}
Now, we claim that the second term converges to zero: 
  \begin{equation}
  \label{equation:term2}
     \lim_{n \rightarrow \infty} \frac{1}{n^{1-2p}}  p^2 \sum_{k=1}^{\lfloor nt \rfloor} \bh{a}_{k}^2  \left( \frac{\hat{S}_{k-1}}{k-1} \right)^2 = 0 \quad \text{ a.s. } 
 \end{equation}
 Indeed,  by (\ref{equation:asymptSeries}) it suffices to notice that by Theorem  \ref{thm:LLNDiffusive}, we have
 \begin{equation*}
     \lim_{k \rightarrow \infty} \frac{ \hat{S}_k}{k} = 0 \quad \text{a.s.}
 \end{equation*}
 since we recall that by our standing assumptions $X$ is both bounded and centered. Finally, we claim that for the last term, the following limit holds:
 \begin{equation}
 \label{equation:term3}
 \lim_{n \rightarrow \infty}
     \frac{1}{n^{1-2p}}  p \sum_{k=1}^{\lfloor nt \rfloor} \bh{a}_{k}^2  \frac{{\hat{V}}_{k-1} }{(k-1)}  =  \frac{{ \sigma^2 }}{1-2p} t^{1-2p} p \quad \text{ a.s.}
 \end{equation}
 In that direction, notice that $(\hat{V}_n)_{n \in \mathbb{N}}$ is the reinforced version of the (non-centered) random walk  
 \begin{equation*}
     V_n = {X}_1^2 + \dots + {X}^2_n, \quad \quad  n \in \mathbb{N}
 \end{equation*}
 with mean $\EE(\hat{X}_i^2) = \EE({X}_i^2)=\sigma^2$. In order to work with a centered reinforced random walk, we introduce 
 \begin{align} \label{equation:reductionToNRRW}
     \hat{W}_n &= \left( \hat{X}_1^2 - \EE(X^2) \right) + \dots + \left( \hat{X}^2_n- \EE(X^2) \right) \\
     &= \hat{Y_1} + \dots + \hat{Y_n}
 \end{align}
 with an obvious  notation. This is the step-reinforced version of the random walk with typical step distributed as $X^2 - \EE(X^2)$, which is centered and bounded. This allows us to write for each $k \in \mathbb{N}$
 \begin{equation*}
     \hat{V}_k = \hat{W}_k + k \cdot \EE(X^2) = \hat{W}_k + k \cdot \sigma^2  
 \end{equation*}
 and by replacing in (\ref{equation:term3}) and the law of large numbers (Theorem \ref{thm:LLNDiffusive}), applied  to the centered reinforced random walk $(\hat{W}_n)_{n \in \mathbb{N}}$, we obtain:
 \begin{align*}
    &\lim_{n \rightarrow \infty} \, \, 
     \frac{1}{n^{1-2p}}  p \sum_{k=1}^{\lfloor nt \rfloor} \bh{a}_{k}^2  \frac{{\hat{V}}_{k-1} }{(k-1)}  \\
    & \hspace{20mm} \, \, =  \lim_{n \rightarrow \infty} \left(
     \frac{1}{ n^{1-2p}}  p \sum_{k=2}^{\lfloor nt \rfloor} \bh{a}_{k}^2  \frac{ \hat{W}_{k-1}}{(k-1)}  + \frac{\sigma^2}{ n^{1-2p}}  p \sum_{k=2}^{\lfloor nt \rfloor} \bh{a}_{k}^2 \right) \\
     & \hspace{20mm} \, \, = \frac{\sigma^2}{1-2p} t^{1-2p} p.
 \end{align*}
 Now, combining (\ref{equation:term1}), (\ref{equation:term2}) and (\ref{equation:term3}) we conclude that 
 \begin{equation*}
     \lim_{n \rightarrow \infty}  \langle \hat{N}^{(n)}, \hat{N}^{(n)} \rangle_t = \frac{\sigma^2}{1-2p}t^{1-2p}. 
 \end{equation*}
\item By (\ref{formula:TheQuadraticVariationBalanced}),  
\begin{align*}
    &\langle \check{N}^{(n)},\check{N}^{(n)} \rangle_t \\
    &\quad \quad =  \frac{1}{n^{1+2p}} \left( \sigma^2 + \sum_{k=2}^{\lfloor nt \rfloor} \bc{a}^2_k \left( (1-p) \sigma^2 - p^2 \left( \frac{\check{S}_{k-1} }{k-1} \right)^2 + p \frac{\hat{V}_{k-1}}{k-1} \right) \right)  
\end{align*}
 and we now study the convergence of the normalised series in the previous expression. By (\ref{formula:asyptoticSeriesCounterbalanced}),  the first term converges towards 
\begin{equation*}
     \lim_{n \rightarrow \infty}  \sigma^2 \frac{(1-p) }{n^{1+2p}} \sum_{k=2}^{\lfloor nt \rfloor}\bc{a}^2_k    =  \sigma^2   \frac{1-p}{1+2p} t^{1+2p} =  \sigma^2  (1-p) \int_0^t s^{2p} \mathrm{d}s.
\end{equation*}
Turning our attention to the second term, we recall from Theorem 1.1 in \cite{BertoinCounterbalancing} that $(\check{S}_n)$ satisfies a law of large numbers: 
\begin{equation*}
    \lim_{n \rightarrow \infty}  \frac{1}{n} \check{S}_n = (1-p) \frac{m}{1+p} =0 \quad \text{ in probability.}
\end{equation*}
Since $\|X \|_\infty< \infty$, $n^{-1}|\check{S}_n| \leq \|X \|_\infty$ and hence the convergence holds in $L^2(\mathbb{P})$. This remark paired with the asymptotic behaviour of the series (\ref{formula:asyptoticSeriesCounterbalanced}) yields:
\begin{equation*}
    \lim_{n \rightarrow \infty} \frac{1}{n^{1+2p}} \sum_{k=2}^{\lfloor nt \rfloor} \bc{a}^2_k \left( \frac{\check{S}_{k-1}}{k-1} \right)^2 =0 \quad \quad \text{in } L^1(\mathbb{P})  \text{  for all } t \geq 0
\end{equation*}
and a fortiori in probability. Finally,  we claim that 
\begin{equation} \label{equation:counterbalancedEq2}
    \lim_{n \rightarrow \infty}  \frac{p}{n^{1+2p}} \sum_{k=2}^{\lfloor nt \rfloor}  \bc{a}^2_k \left( \frac{\hat{V}_{k-1}}{k-1} \right) =  \sigma^2  p \int_0^t s^{2p} \mathrm{d}s \quad \text{in probability. }
\end{equation}
We start assuming that $p<1$, and  proceeding as in (\ref{equation:reductionToNRRW}), we set: 
\begin{align} 
     \hat{W}_n &= \left( \hat{X}_1^2 - \EE(X^2) \right) + \dots + \left( \hat{X}^2_n- \EE(X^2) \right) \\
     &= \hat{Y_1} + \dots + \hat{Y_n}
 \end{align}
It follows that $\hat{V}_n = \hat{W}_n + n  \sigma^2 $ , where $\hat{W}_n$ is a centred  noise reinforced random walk whose steps have the law of $X^2-\EE(X^2)$ with memory parameter $p$. Since $p \in (0,1)$, we recall  from Corollary \ref{thm:WeakLLN} that for all regimes, we have $n^{-1}\hat{W}_n \rightarrow 0$ in $L^1(\PP)$. As a consequence, 
  \begin{equation*}
      \lim_{n \rightarrow \infty}   \frac{p}{n^{1+2p}} \sum_{k=2}^{\lfloor nt \rfloor}  \bc{a}^2_k \left( \frac{\hat{W}_{k-1}}{k-1} \right) = 0 \quad \quad \text{ in } L^1(\PP),
  \end{equation*}
  and a fortiori in probability.  We deduce that 
 \begin{align} 
     &\lim_{n \rightarrow \infty}  \frac{p}{n^{1+2p}} \sum_{k=2}^{\lfloor nt \rfloor}  \bc{a}^2_k \left( \frac{\hat{V}_{k-1}}{k-1} \right) \nonumber \\
     &\hspace{20mm} = \lim_{n \rightarrow \infty}  \left(  \frac{p}{n^{1+2p}} \sum_{k=2}^{\lfloor nt \rfloor}  \bc{a}^2_k \left( \frac{\hat{W}_{k-1}}{k-1} \right) + \frac{  \sigma^2  p}{n^{1+2p}} \sum_{k=2}^{\lfloor nt \rfloor}  \bc{a}^2_k \right) \\
     &\hspace{20mm} =   \sigma^2 \frac{p}{1+2p} t^{1+2p} \label{limit:previouslyStablishedLimit}
 \end{align}
in probability,  which proves (\ref{equation:counterbalancedEq2}). If $p=1$, by hypothesis $X$ takes its values in $\{ -1, 1\}$ and $\check{V}_{k-1} ={k-1}$, yielding that the previously established limit (\ref{limit:previouslyStablishedLimit})  still holds, replacing $1-p$ by $0$. Notice however that if we allowed $X$ to take arbitrary values, we can no longer proceed as we just did since in that case, $\hat{V}_n$ is a straight line  with random slope:
 \begin{equation*}
     \check{V}_n = n \check{X}_1^2.
 \end{equation*}
 Putting all pieces together, we obtain  (ii).
 \item Recalling that $\hat{X}_k = X_k \mathbf{1}_{\{ \epsilon_k=0 \}} + \hat{X}_{U[k-1]} \mathbf{1}_{\{ \epsilon_k=1 \}}$,   and from independence of $X_k$, $\epsilon_k$ and $U[k-1]$ from $\mathcal{F}_{k-1}$, we get for $k \geq 2$  
\begin{align*}
    \EE ( \Delta \hat{M}_{k} X_{k} \mid \mathcal{F}_{k-1} ) 
    &= \EE \left(  ( \hat{S}_{k-1} ( \bh{a}_k - \bh{a}_{k-1} ) + \hat{X}_k \bh{a}_k ) X_k    \mid \mathcal{F}_{k-1} \right) \\
    &= \bh{a}_k \EE \left(  \left(  X_k \mathbf{1}_{\{ \epsilon_k=0 \}} + \hat{X}_{U[k-1]} \mathbf{1}_{\{ \epsilon_k=1 \}} \right) X_k    \mid \mathcal{F}_{k-1} \right)  \\
    &= \bh{a}_k (1-p) \EE(X^2) + \sum_{j=1}^{k-1}  \EE \left(  X_k X_j \mathbf{1}_{\{ U[k-1] =j , \epsilon_k=1 \}} \mid \mathcal{F}_{k-1} \right) \\
    &=\bh{a}_k (1-p)  \sigma^2 
\end{align*}
since the steps are centered, while for $k=1$ we simply get $\EE(\hat{M}_1X_1) =  \sigma^2$. From here,  we deduce 
\begin{equation*}
    \langle \hat{N}^{(n)}, N^{(n)} \rangle_t 
    = n^{p-1} \sum_{k=1}^{\lfloor nt  \rfloor}  \EE ( \Delta M_{k} X_{k} \mid \mathcal{F}_{k-1} ) =  \sigma^2  (1-p) n^{p-1} \left(  (1-p)^{-1} + \sum_{k=2}^{\lfloor nt \rfloor } \bh{a}_k \right) 
\end{equation*}
 and from the convergence
\begin{equation*}
    \lim_{n \rightarrow \infty}  n^{p-1} \sum_{k=2}^n \bh{a}_k = (1-p)^{-1}
\end{equation*}
we conclude  
\begin{equation*}
    \lim_{n \rightarrow \infty} \langle \hat{N}^{(n)}, N^{(n)} \rangle_t =  \sigma^2  (1-p) \lim_{n \rightarrow \infty}  {n^{p-1}} \sum_{k=2}^{\lfloor nt \rfloor } \bh{a}_k = t^{1-p} =  \sigma^2  (1-p) \int_0^t s^{-p} \mathrm{d}s. 
\end{equation*}
\item Recalling that in the counterbalanced case  $\check{X}_k = X_k \mathbf{1}_{\{ \epsilon_k=1 \}}  -\check{X}_{U[k-1]} \mathbf{1}_{\{ \epsilon_k=0 \}}$,  we deduce from similar arguments as in the reinforced case, 
\begin{align*}
    \EE ( \Delta \check{M}_{k} X_{k} \mid \mathcal{F}_{k-1} ) 
    &= \EE \left(  ( \check{S}_{k-1} ( \bc{a}_k - \bc{a}_{k-1} ) + \check{X}_k \bc{a}_k ) X_k    \mid \mathcal{F}_{k-1} \right) \\
    &= \bc{a}_k\EE \left(  \left(  X_k \mathbf{1}_{\{ \epsilon_k=1 \}}  -\check{X}_{U[k-1]} \mathbf{1}_{\{ \epsilon_k=0 \}} \right) X_k    \mid \mathcal{F}_{k-1} \right)  \\
    &= \bc{a}_k \cdot (1-p) \EE(X^2) - \sum_{j=1}^{k-1}  \EE \left(  X_k X_j \mathbf{1}_{\{ U[k-1] =j , \epsilon_k=0 \}} \mid \mathcal{F}_{k-1} \right) \\
    &= \bc{a}_k \cdot (1-p) \sigma^2. 
\end{align*}
Notice that if $p=1$ the argument still holds and hence the above quantity is null for $k \geq 1$. It follows that
\begin{equation*}
    \langle \check{N}^{(n)}, N^{(n)} \rangle_t 
    =  n^{-(1+p)} \sum_{k=1}^{\lfloor nt  \rfloor}  \EE ( \Delta \check{M}_{k} X_{k} \mid \mathcal{F}_{k-1} ) =  \sigma^2  (1-p)  \cdot  n^{-(1+p)} \sum_{k=1}^{\lfloor nt \rfloor } \bc{a}_k
\end{equation*}
 and from the convergence
\begin{equation*}
    \lim_{n \rightarrow \infty}  n^{-(1+p)} \sum_{k=1}^n \bc{a}_k = (1+p)^{-1}
\end{equation*}
we conclude  
\begin{equation*}
   \sigma^{-2}  \lim_{n \rightarrow \infty} \langle \check{N}^{(n)} , N^{(n)} \rangle_t =  (1-p)  \lim_{n \rightarrow \infty}  {n^{-(1+p)}} \sum_{k=1}^{\lfloor nt \rfloor } \bc{a}_k = \frac{1-p}{(1+p)}  t^{1+p} = (1-p) \int_0^t s^{p} \mathrm{d}s. 
\end{equation*}
Finally  if $p=1$, we clearly have $\lim_{n \rightarrow \infty} \langle \check{N}^{(n)},N^{(n)} \rangle_t =0$.
\item Notice that 
\begin{align*}
    \EE ( \Delta \check{M}_k \Delta \hat{M}_k \mid \mathcal{F}_{k-1} ) 
    & = \EE \left( 
     ( \hat{S}_{k-1} ( \bh{a}_k - \bh{a}_{k-1} ) + \hat{X}_k \bh{a}_k )
      ( \check{S}_{k-1} ( \bc{a}_k - \bc{a}_{k-1} ) + \check{X}_k \bc{a}_k )
    \mid \mathcal{F}_{k-1} \right) \\
    & \hspace{-15mm} =  \hat{S}_{k-1} ( \bh{a}_k - \bh{a}_{k-1} ) \check{S}_{k-1} ( \bc{a}_k - \bc{a}_{k-1} ) 
    +  \hat{S}_{k-1} ( \bh{a}_k - \bh{a}_{k-1} ) \EE (\check{X}_k | \mathcal{F}_{k-1}) \bc{a}_k  \\
    & \hspace{20mm} +  \check{S}_{k-1} ( \bc{a}_k - \bc{a}_{k-1} ) \EE (\hat{X}_k \mid \mathcal{F}_{k-1}) \bh{a}_k 
    + \EE ( \check{X}_k \hat{X}_k  \mid \mathcal{F}_{k-1}) \bh{a}_k \bc{a}_k \\
    &\hspace{-15mm}  =: P^{(a)}_k + P^{(b)}_k + P^{(c)}_k + P^{(d)}_k.
\end{align*}
where the notation  was assigned in order of appearance. We write,
\begin{equation*}
    \langle \check{N}^{n}, \hat{N}^{n} \rangle_t 
    = n^{-1} \sum_{k=1}^{\lfloor nt \rfloor } \left( P^{(a)}_k + P^{(b)}_k + P^{(c)}_k + P^{(d)}_k \right) 
\end{equation*}
and study the asymptotic behaviour of these four terms individually. In that direction, we recall from (\ref{formula:conditionalRstep}) and (\ref{formula:conditionalCstep}) the identities $\EE (\hat{X}_k \mid  \mathcal{F}_{k-1}) = p \hat{S}_{k-1}/(k-1)$,  $\EE (\check{X}_k \mid  \mathcal{F}_{k-1}) = -p \check{S}_{k-1}/(k-1)$ as well as from (\ref{equation:asymptoticDiferenceHat}) the asymptotic behaviour $(\bh{a}_k - \bh{a}_{k-1}) \sim p k^{-(p+1)}$ while a similar computation yields  $(\bc{a}_k - \bc{a}_{k-1}) \sim p k^{p-1}$.
\begin{itemize}
    \item We first show that 
    \begin{align*}
        \lim_{n \to \infty} n^{-1} \sum_{k=1}^{ \lfloor nt \rfloor} P_k^{(c)}=0 \quad \text{a.s.}.
    \end{align*}From the identities and  asymptotic estimates we just recalled, we have 
\begin{align*}
     \check{S}_{k-1} ( \bc{a}_k - \bc{a}_{k-1} ) \EE (\hat{X}_k \mid  \mathcal{F}_{k-1}) \bh{a}_k 
     = \check{S}_{k-1} ( \bc{a}_k - \bc{a}_{k-1} )p \frac{\hat{S}_{k-1}}{k-1} \bh{a}_k 
     \sim \frac{\check{S}_{k-1}}{k} k^p  p^2 \frac{\hat{S}_{k-1}}{k-1} \bh{a}_k
\end{align*}
and since $\bh{a}_k \sim k^{-p}$,  we have for some constant $C$ large enough, 
\begin{equation*}
   n^{-1} \biggl| \sum_{k=1}^{\lfloor nt \rfloor } P_k^{(c)} \biggr|
     \leq   n^{-1} C \sum_{k=1}^{\lfloor nt \rfloor } \biggl| \frac{\check{S}_{k-1} }{k} \frac{\hat{S}_{k-1} }{k-1} \biggr| \leq n^{-1} C \| X \|_\infty \sum_{k=1}^{\lfloor nt \rfloor } \biggl| \frac{\hat{S}_{k-1} }{k-1} \biggr|
\end{equation*}
which converges a.s. towards $0$ as $n \uparrow \infty$ by Lemma \ref{thm:LLNDiffusive}.
\item Next, since \begin{align*}
     \hat{S}_{k-1} ( \bh{a}_k - \bh{a}_{k-1} ) \EE (\check{X}_k \mid  \mathcal{F}_{k-1}) \bc{a}_k 
     = -\hat{S}_{k-1} ( \bh{a}_k - \bh{a}_{k-1} )p \frac{\check{S}_{k-1}}{k-1} \bc{a}_k 
     \sim -\frac{\hat{S}_{k-1}}{k} k^{-p}  p^2 \frac{\check{S}_{k-1}}{k-1} \bc{a}_k
\end{align*} we can follow exactly the same line of reasoning in order to establish
\begin{align*}
      \lim_{n \to \infty} n^{-1} \sum_{k=1}^{ \lfloor nt \rfloor} P_k^{(b)}=0 \quad \text{a.s.}.
\end{align*}
\item Since  
\begin{equation*}
    (\bh{a}_k - \bc{a}_{k-1})(\bh{a}_k - \bc{a}_{k-1}) \sim p^2 k^{-2}
\end{equation*}
we deduce that 
\begin{equation*}
    \hat{S}_{k-1} ( \bh{a}_k - \bh{a}_{k-1} ) \check{S}_{k-1} ( \bc{a}_k - \bc{a}_{k-1} ) \sim    \hat{S}_{k-1} \check{S}_{k-1} k^{-2}
\end{equation*}
and we conclude as before that 
\begin{align*}
      \lim_{n \to \infty} n^{-1} \sum_{k=1}^{ \lfloor nt \rfloor} P_k^{(a)}=0 \quad \text{a.s.}.
\end{align*}
\item Finally, since by definition 
\begin{equation*}
    \hat{X}_k = X_k \mathbf{1}_{\{ \epsilon_k = 0 \}} + \hat{X}_{U[k-1]}\mathbf{1}_{\{ \epsilon_k = 1 \}}, \qquad 
    \check{X}_k = X_k \mathbf{1}_{\{ \epsilon_k = 0 \}} - \check{X}_{U[k-1]}\mathbf{1}_{\{ \epsilon_k = 1 \}}
\end{equation*}
we have 
\begin{align*}
    \bh{a}_k \bc{a}_k E ( \check{X}_k \hat{X}_k  \mid \mathcal{F}_{k-1}) 
    &= \bh{a}_k \bc{a}_k E ( {X}_k ^2 \mathbf{1}_{\{ \epsilon_k = 0 \}} \mid \mathcal{F}_{k-1})
    - \bh{a}_k \bc{a}_k E ( \check{X}_{U[k-1]} \hat{X}_{U[k-1]} \mathbf{1}_{\{ \epsilon_k = 1 \}} \mid \mathcal{F}_{k-1}) \\
    & = \bh{a}_k \bc{a}_k (1-p) \sigma^2 
    - \bh{a}_k \bc{a}_k \sum_{j=1}^{k-1}  E ( \check{X}_{j} \hat{X}_{j} \mathbf{1}_{\{ \epsilon_k = 1 , U[k-1] = j \}} \mid \mathcal{F}_{k-1}).
\end{align*}
Since on one hand, $\check{X}_j,\hat{X}_j$ for $j < k$ are $\mathcal{F}_{k-1}$ measurable while $\epsilon_k, U[k-1]$ are independent of $\mathcal{F}_{k-1}$, denoting as $\check{G}$ the counterbalanced random walk made from the i.i.d. sequence $X^2_1, X^2_2, \dots$ from the same instance of the reinforcement algorithm,  we deduce 
\begin{align*}
 P^{(d)}_k 
 =   \bh{a}_k \bc{a}_k \left(  (1-p) \sigma^2 
    -\frac{1}{k-1} p \sum_{j=1}^{k-1}  \check{X}_{j} \hat{X}_{j} \right) 
 =  \bh{a}_k \bc{a}_k \left(  (1-p) \sigma^2 
    - p  \frac{\check{G}_{k-1}}{k-1}\right)
\end{align*}
and since  $\hat{a}_k \check{a}_k \rightarrow 1$ as $k \to \infty$, the problem boils down to studying the convergence as $n \uparrow \infty$ of 
\begin{equation*}
 \frac{1}{n} \sum_{k=1}^{\lfloor nt\rfloor} \left( (1-p)\sigma^2  
- p\frac{\check{G}(k-1)}{k-1}  \right).
\end{equation*}
The first term obviously converges towards $t(1-p) \sigma^2$ and we turn our attention to the second one. In that direction, at each $k$ we decompose $\check{G}(k) = \check{G}_1(k) + \sum_{j=1}^{k} \check{G}_j(k)$ where $\check{G}_i(k)$ consists exclusively of the sum of the steps that have been repeated $i$-times at step $k$. Since the steps have mean $m = \sigma^2$  we get  from Lemma 4.1 in \cite{BertoinCounterbalancing} (beware that $p$ here is $1-p$ in \cite{BertoinCounterbalancing})  that: 
\begin{enumerate}
    \item $\lim_{k \rightarrow \infty} k^{-1} \check{G}_1(k) = \sigma^2 (1-p)/(1+p)$ a.s.
    \item $ \lim_{k \rightarrow \infty}  k^{-1} \sum_{j=2}^k |\check{G}_j(k)| = 0$ in probability.
\end{enumerate}
Notice that   (b)   holds in $L^1(\PP)$ too since  $k^{-1} \sum_{j=2}^k |G_j(k)| \leq \| X \|_\infty$, as there are at most $k-1$ repeated steps at time $k$. Hence, 
\begin{equation*}
    - \lim_{n \rightarrow \infty}  \frac{p}{n} \sum_{k=1}^{\lfloor nt \rfloor} \frac{ \sum_{j=2}^{k-1} \check{G}_j(k-1)}{k-1} \rightarrow 0 \quad \quad \text{ in } L^1(\PP)
\end{equation*}
 and finally, we deduce that 
\begin{align*}
    -  \lim_{n \rightarrow \infty} \frac{p}{n}\sum_{k=1}^{\lfloor nt \rfloor} \frac{\check{G}(k-1)}{k} 
    = - \lim_{n \rightarrow \infty} \frac{p}{n}  \sum_{k=1}^{\lfloor nt \rfloor} \frac{\check{G}_1(k-1)}{k}  + \frac{\sum_{j=2}^k \check{G}_j(k-1)}{k} = - p \sigma^2 t \frac{1-p}{1+p} 
\end{align*}
in probability, by the almost sure convergence in Cesaro-mean. Putting all  pieces together, we conclude that  the following convergence holds in probability: 
\begin{equation*}
    \lim_{n \rightarrow \infty} \frac{1}{n} \sum_{k=1}^{\lfloor nt \rfloor} P_k^{(d)} = t(1-p) \sigma^2 - p \sigma^2 t\frac{1-p}{1+p}.
\end{equation*}
\end{itemize}
Bringing all our calculations above together we conclude  the convergence in probability, 
\begin{equation*}
    \lim_{n \rightarrow \infty} \langle \hat{N}^n, \check{N}^{n} \rangle_t = \sigma^2 (1-p) t  - p \sigma^2 (1-p)(1+p)^{-1} t.
\end{equation*}
\end{enumerate}
This concludes the proof of the lemma.
\end{proof}
With this, we conclude the proof of Theorem \ref{thm:ConvergenceTripletDifusive} when $X$ is bounded with an appeal to Lemma \ref{lemma:asympnegliblejumps}, Lemma \ref{lemma:quadvarconv} and the MFCLT (Theorem \ref{thm:FCLT}). 
\section{Reduction to the case of bounded steps.} \label{section:reductionArgument}
In this section, we shall only assume that the typical step $X \in L^2(\PP)$ of the step-reinforced random walk $\hat{S}$ is centred and no longer that it is bounded. We shall complete the proof of Theorem \ref{thm:ConvergenceTripletDifusive} by means of the truncation argument reminiscent to the one of Section 4.3 in \cite{BertoinUniversality}.
\subsection{Preliminaries}
The reduction argument relies on the following lemma taken from \cite{Processes}, that we state for the  reader's convenience: 
\begin{lem}[Lemma 3.31 in Chapter VI of \cite{Processes}] \label{lemma:reductionLemma} \mbox{}\\
Let $(Z^n)$ be a sequence of $d$-dimensional rcll (càdlàg) processes and suppose that 
\begin{equation*}
    \forall N >0, \quad \forall \epsilon >0 \quad \quad \lim_{n \rightarrow \infty} \PP\left( \sup_{s \leq N} |Z_s^n| > \epsilon \right)=0.
\end{equation*}
If $(Y^n)$ is another sequence of $d$-dimensional rcll processes with $Y^n \Rightarrow Y$ in the sense of Skorokhod, then $Y^n + Z^n \Rightarrow Y$ in the sense of Skorokhod. 
\end{lem}

Finally, we will need the following lemma concerning convergence on metric spaces: 
\begin{lem} \label{lemma:theReductionLemma} Let $(E,d)$ be a metric space and consider  $(a_n^{(m)} \, : \, m,n \in \mathbb{N})$ a family of sequences, with  $a_n^{(m)} \in E$ for all $n, m \in \mathbb{N}$. Suppose further that the following conditions are satisfied:
\begin{enumerate}[(i)]
    \item For each fixed $m$, $a_n^{(m)} {\rightarrow} a_\infty^{(m)}$ as $n \uparrow \infty$ for some element $a_\infty^{(m)} \in E$.
    \item $a_\infty^{(m)} {\rightarrow} a_{\infty}^{(\infty)}$ as  $m \uparrow \infty$, for some $a_\infty^{(\infty)} \in E$.
\end{enumerate}
Then, there exists a non-decreasing  subsequence $(b(n))_{n}$ with $b(n) \rightarrow  \infty$ as $n \uparrow \infty$,  for which the following convergence holds:
\begin{equation*}
    a_n^{(b(n))} {\rightarrow} a_\infty^{(\infty)} \quad \text{ as } n \uparrow \infty.
\end{equation*}
\end{lem}
\begin{proof}
Since the sequence $(a_\infty^{(m)})_m$ converges, we can find an increasing subsequence $m_1 \leq m_2 \leq \dots $ satisfying 
\begin{equation*}
    d(a_\infty^{(m_k)} , a_\infty^{(m_{k+1})}  ) \leq 2^{-k} \quad \quad \text{ for each } k \in \mathbb{N}.
\end{equation*}
Moreover, since for each fixed $m_k$ the corresponding sequence $(a_n^{(m_k)})_n$ converges, there exists a strictly increasing sequence $(n_k)_k$ satisfying that, for each $k$, 
\begin{align*}
     d(a_i^{(m_k)} , a_\infty^{(m_k)} ) \leq 2^{-k} \quad \quad \text{ for all } i \geq n_k.
\end{align*}
Now, we set
for $n < n_1$, $b(n) := m_1$ and for $k \geq 1$, 
$b(n) := m_k$ if $n_k\leq n < n_{k+1}$ and we claim $(a_n^{b(n)})_n$ is the desired sequence. Indeed, it suffices to observe that for $n_k \leq n < n_{k+1}$,
\begin{equation*}
    d(a_n^{(b(n))} , a_\infty^{(\infty)}) 
    = d(a_n^{(m_k)} , a_\infty) 
    \leq d(a_n^{(m_k)} , a_\infty^{(m_k)}) + d(a_\infty^{(m_k)}, a_\infty) \leq  2^{-k} + \sum_{i=k}^\infty 2^{-i}.
\end{equation*}
\end{proof}
\subsection{Reduction argument}
Recall that we are assuming that the typical step is centred.  During the course of this section we will use that the truncated versions of the counterbalanced and noise reinforced random walks are still  counterbalanced (resp. noise reinforced) random walks.

Indeed, notice that if $(\check{S}_n)$ and $(\hat{S}_n)$ have been built from the i.i.d. sequence $(X_n)_{n \geq 1}$  by means of the negative-reinforcement and positive-reinforcement algorithms described in the introduction,  splitting each $X_i$ for $i \in \mathbb{N}$ as 
\begin{equation*}
    X_i = X_i^{\leq K}  + X_i^{ > K} 
\end{equation*}
where respectively, 
\begin{align*}
    X_i^{\leq K} &:= X_i \mathbf{1}_{\{ |X_i| \leq K \}} - \EE \left( X_i \mathbf{1}_{\{ |X_i| \leq K \}} \right)  \\
    X_i^{ > K}  &:=  X_i \mathbf{1}_{\{ |X_i| > K \}} - \EE(X_i \mathbf{1}_{\{|X_i| > K \}}),
\end{align*}
 yields a natural decompositions for $(\check{S}_n)$ and $(\hat{S}_n)$ in terms of two counterbalanced (reps. noise reinforced) random walks:
\begin{align*}
     \check{S}_n =  \check{S}^{ \leq K}_n + \check{S}^{> K}_n, \hspace{20mm}
     \hat{S}_n =  \hat{S}^{ \leq K}_n + \hat{S}^{> K}_n
\end{align*}
where now $(\check{S}^{ \leq K}_n)$, $(\check{S}^{> K}_n )$ are counterbalanced versions with typical step centred and distributed respectively as 
\begin{equation}
     X^{\leq K} := X \mathbf{1}_{\{ |X| \leq K \}} - \EE \left( X \mathbf{1}_{\{ |X| \leq K \}} \right)
\end{equation}
and 
\begin{equation}
    X^{ > K}  :=  X \mathbf{1}_{\{ |X| > K \}} - \EE(X \mathbf{1}_{\{|X| > K \}}), 
\end{equation}
an  analogue statement holding in the reinforced case for $(\hat{S}^{ \leq K}_n)$, $(\hat{S}^{> K}_n )$.  Moreover, $X^{\leq K}$ is centred with variance ${\sigma}^2_K$ and $\sigma^2_K \rightarrow \sigma^2$ as $K \nearrow  \infty $ while the variance of $X^{> K}$ that we denote by $\eta^2_K$, converges towards zero as $K \uparrow \infty$. We will also write the respective  truncated random walk as 
\begin{equation*}
    S^{\leq K}_n = X_1^{\leq K} + \dots + X_n^{\leq K} 
    \hspace{20mm} 
    S^{> K}_n = X_1^{> K} + \dots + X_n^{>K} \quad \quad  \quad n \geq 1. 
\end{equation*}
Notice that $(S^{\leq K})$,  $(\hat{S}^{ \leq K}_n)$ and $(\check{S}^{ \leq K}_n)$  have  now  \textit{bounded} steps, allowing us to apply the result established in Section \ref{section:proofBoundedSteps} to this triplet.
\begin{rem}
We point out that while $(\hat{S}_n^{\leq K})$ can be simply obtained by considering the NRRW made from the steps $X_i 1_{\{ |X_i| \leq K \}}$, $i \geq 1$  and substracting $n \EE(X 1_{\{ X \leq K \}})$ at the n-th step for each $n\geq 1$, and hence yielding a NRRW with steps given by 
\begin{equation*}
    \hat{X}_i1_{\{ |\hat{X}_i| \leq K \}} - \EE(X 1_{\{|X| \leq K \}}),
\end{equation*}
for the counterbalanced case we need to subtract the counterbalanced random walk issued from the constants $\EE( X_i 1_{\{|X_i| \leq K \}})$, $i \geq 1$ , which in contrast with the reinforced case, is a \textit{process} on its own right because of the sign swap.
\end{rem}

For each $k$, write as $N^{n,K}$ $\hat{N}^{n,K}$ and $\check{N}^{n,K}$ the corresponding martingales as defined in (\ref{equation:martingaleLimit1}) relative to  $S^{\leq K}$,  $\hat{S}^{\leq K}$ and $\check{S}^{\leq K}$ respectively. An application of Theorem \ref{thm:ConvergenceTripletDifusive} in the bounded case yields for every $K$, that 
\begin{equation*}
    \left(  N^{n,\leq K}_t, \hat{N}^{\leq K,n}_t , \check{N}^{\leq K,n}_t  \right)_{t \in \mathbb{R}^+} 
     \implies 
    \left( \sigma_K B,\,  \sigma_K \int_0^t s^{-p} d\beta^r_s, \,  \sigma_K \int_0^t s^{p} d\beta^c_s \right).   \label{convergence:DiffusiveBounded} 
\end{equation*}
However recalling the asymptotic behaviour  ${n^p}{  \bh{a}_{\lfloor nt \rfloor}} \sim t^{-p} \text{ as } n \rightarrow \infty$ and the definition of $N^{n, \leq K}$, we deduce that 
\begin{align*}
    \left(  \frac{{S}^{\leq K}(\lfloor nt \rfloor)}{\sqrt{n} } , \frac{\hat{S}^{\leq K}( \lfloor nt \rfloor )}{\sqrt{n}} , \check{N}^{\leq K, n}_t   \right)_{t \in \mathbb{R}^+} 
    & \implies 
    \left( \sigma_K B ,\,  \sigma_K t^p \int_0^t s^{-p} d \beta^r_s  , \, \sigma_K  \int_0^t s^{p} d\beta^c_s \right).    \label{convergence:DiffusiveBounded} 
\end{align*}
Since as $K \uparrow \infty$, the right hand side converges weekly towards $(\sigma B_t, \sigma t^p \int_0^t s^{-p} d \beta^r_s , \sigma \int_0^t s^{p} d\beta^c_s)$ and the convergence in distribution is metrisable, by Lemma \ref {lemma:theReductionLemma} there exists a slowly increasing sequence  converging towards infinity that we denote as  $(K(n):  n \geq 1)$, satisfying that, as $n \uparrow \infty$,  
\begin{align*}
    \left(  \frac{{S}^{\leq K(n)}(\lfloor nt \rfloor)}{\sqrt{n} } , \frac{\hat{S}^{\leq K(n)}( \lfloor nt \rfloor )}{\sqrt{n}} , \check{N}^{\leq K(n), n}_t   \right)_{t \in \mathbb{R}^+} 
    & \implies 
    \left( \sigma B , 
    \sigma t^p \int_0^t s^{-p} d \beta_s , 
    \sigma  \int_0^t s^{p} d\beta_s \right).  
\end{align*}
On the other hand, for each $n$ we can clearly decompose 
\begin{align*}
 \left( \frac{ S (\lfloor nt \rfloor) }{\sqrt{n}}, \frac{\hat{S}( \lfloor nt \rfloor )}{\sqrt{n}}, \check{N}^{n}  \right)  
    &= 
 \left(  \frac{{S}^{\leq K(n)}(\lfloor nt \rfloor)}{\sqrt{n} } , \frac{\hat{S}^{\leq K(n)}( \lfloor nt \rfloor )}{\sqrt{n}} , \check{N}^{\leq K(n), n}_t   \right) \\
& \hspace{15mm} + \left(  \frac{{S}^{> K(n)}(\lfloor nt \rfloor)}{\sqrt{n} } , \frac{\hat{S}^{> K(n)}( \lfloor nt \rfloor )}{\sqrt{n}} , \check{N}^{> K(n), n}_t   \right),
\end{align*}
and in order to apply Lemma \ref{lemma:reductionLemma} we need the following lemma: 
\begin{lem} For any   sequence $(K(n): n \geq 1)$  increasing towards infinity 
the following limits hold: 
\begin{enumerate}[(i)]
    \item $\displaystyle \lim_{n \rightarrow \infty} \frac{1}{n} \EE \left( \sup_{k \leq nt} \left|{S}^{> K(n)}_k\right|^2 \right) = 0.$
    \item $\displaystyle \lim_{n \rightarrow \infty} \frac{1}{n} \EE \left( \sup_{k \leq nt} \left|\hat{S}^{> K(n)}_k \right|^2 \right) = 0, \qquad \text{for } p \in (0,1/2).$
    \item $\displaystyle \lim_{n \rightarrow \infty} \PP \left( \sup_{s \leq T} \left|\check{N}_s^{n, >K(n)} \right|^2 \geq \epsilon \right) = 0,$ \qquad for every $\epsilon >0$ and $p \in (0,1).$
\end{enumerate}
\end{lem} 
\begin{proof} Recall that we denoted by $\eta^2_K$ the variance of  $X^{> K}$.
\begin{enumerate}[(i)]
    \item By Doob's inequality and independence of the steps we inmediatly get that 
\begin{equation*}
    \frac{1}{n} \EE \left( \sup_{k \leq nt} |{S}^{> K(n)}_k|^2 \right)  \leq \frac{4}{n}\eta_{K(n)} {\lfloor nt \rfloor}
\end{equation*}
which converges towards $0$ as $n \uparrow \infty$. 
\item From Lemma \ref{lem:maximalIneq} for $0<p<1/2$ we deduce that for any $t>0$,
\begin{align}
    \lim_{n \to \infty} \frac{1}{n} \EE \left( \sup_{k \leq nt} | \hat{S}^{>K(n)}(k)|^2 \right) &\leq c_1 \lim_{n \to \infty} \eta_{K(n)}^2 t =0, \label{ConditionDiffusive}
\end{align}
proving the claim.
\item Doob's maximal inequality yields 
\begin{align*}
    \mathbb{P}\left( \sup_{s \leq T} |\check{N}^{n, > K(n) }_s| \geq \epsilon \right) \leq \epsilon^{-2} \EE \left(\langle \check{N}^{n,> K(n) }, \check{N}^{n,> K(n) }  \rangle_T \right),
\end{align*}
and if we denote  by $\hat{V}^{> (n)}$  the sum of squared steps associated to $(S^{>K(n)})$, notice that  
\begin{align*}
  &\langle  \check{N}^{n, > K(n) }, \check{N}^{n, > K(n) }  \rangle_T \\
  & \hspace{10mm} \leq   \frac{1}{n^{1+2p}} \left(  \eta^2_{K(n)}+ \sum_{k=2}^{\lfloor nT \rfloor} \bc{a}^2_k \left( (1-p) \eta^2_{K(n)} + p \frac{\hat{V}^{> K(n)}_{k-1}}{k-1} \right) \right).  
\end{align*}
Recalling that  $\EE(\hat{V}_{k-1}^{> K(n)}) = (k-1)\eta_{K(n)}^2$, this yields the bound 
\begin{align*}
    \mathbb{P}\left( \sup_{s \leq T} |\check{N}^{n, K(n) }_s| \geq \epsilon \right)
    \leq \epsilon^{-2} \eta_{K(n)}^2  \frac{1}{n^{1+2p}} \left( 1+ \sum_{k=2}^{\lfloor nT \rfloor}\bc{a}^2_k   \right).
\end{align*}
Since on the one hand we have $\eta^2_{K(n)} \rightarrow 0$ as $n \uparrow \infty$ while on the other by (\ref{formula:asyptoticSeriesCounterbalanced}) it holds that  \begin{align*} \limsup_{n \uparrow \infty} n^{-(1+2p)} \sum_{k=2}^{\lfloor nT \rfloor} \bc{a}_k^2 < \infty,
\end{align*}
the desired convergence follows.
\end{enumerate}
This concludes the proof of the lemma.
\end{proof}
Now, recalling the definition of $\check{N}^{n}$, we deduce from Lemma \ref{lemma:reductionLemma} that as $n \uparrow \infty$,
\begin{equation*}
    \left( \frac{1}{\sqrt{n}}S_{\lfloor nt \rfloor}, \frac{1}{\sqrt{n}}\hat{S}_{\lfloor nt \rfloor},   \frac{1}{\sigma \sqrt{ n}} \frac{b_{\lfloor nt \rfloor }}{n^{p}} \check{S}_{\lfloor nt \rfloor } \right)_{t \in \mathbb{R}^+}   
    \implies \left( B_t , t^{p} \int_0^t s^{-p}d \beta^r_s , \int_0^t s^{p} d\beta^c_s  \right)_{t \in \mathbb{R}^+} 
\end{equation*}
and since $b_{\lfloor nt \rfloor}/ n^{p} \sim t^{p}$, we conclude that for any $\delta>0$,  the desired convergence 
\begin{equation*}
    \left( \frac{1}{\sqrt{n}}S_{\lfloor nt \rfloor}, \frac{1}{\sqrt{n}}\hat{S}_{\lfloor nt \rfloor}, \frac{1}{\sigma \sqrt{ n}}  \check{S}_{\lfloor nt \rfloor }  \right)_{t \in [\delta, \infty) }   
    \implies \left( B_t , t^{p} \int_0^t s^{-p}d \beta_s , \, t^{-p}\int_0^t s^{p} d\beta_s  \right)_{t \in [\delta, \infty)},
\end{equation*}
holds away for the origin (this restriction is due to the fact that $t^{-p}$ is unbounded on any neighbourhood of 0). In order to get the convergence on $\mathbb{R}^+$ and finally prove the claimed convergence in Theorem \ref{thm:ConvergenceTripletDifusive}, we proceed as follows: We will only work with the third coordinate, as it is the only one presenting the difficulty. The argument is readily adapted to the triplet. Assume without loss of generality that $\sigma^2 = 1$, fix $\delta >0$ and consider the partition of $[0,\delta]$, with points $\{ \delta 2^{-i} :\,  i = 0, 1, 2,  \dots \}$. Since the sequence $(\bc{a}_k)$ is increasing we obtain, 
\begin{align*}
    \mathbb{P} \left(\sup_{s \in [2^{-(i+1)} \delta , 2^{-i} \delta  ]} \frac{|\check{S}_{\lfloor ns \rfloor}|}{\sqrt{n}}  > \epsilon \right) 
    &\leq 
    \mathbb{P} \left( \sup_{s \in [2^{-(i+1)} \delta , 2^{-i} \delta  ]} \frac{\bc{a}_{\lfloor ns \rfloor}}{\bc{a}_{\lfloor n2^{-(i+1)}\delta  \rfloor}} \frac{|\check{S}_{\lfloor ns \rfloor}|}{\sqrt{n}}  > \epsilon \right) \\
    &\leq 
    \mathbb{P} \left( \sup_{s \in [2^{-(i+1)} \delta , 2^{-i} \delta  ]} {\bc{a}_{\lfloor ns \rfloor}} |\check{S}_{\lfloor ns \rfloor}|  > \epsilon \cdot \bc{a}_{\lfloor n2^{-(i+1)}\delta  \rfloor} {\sqrt{n}} \right)\\
    &= \frac{1}{\epsilon^2 \cdot  n \bc{a}^2_{\lfloor 2^{-(i+1)}  \delta n \rfloor}} \EE \left( \sup_{s \leq 2^{-i} \delta} |\bc{a}_{\lfloor ns \rfloor } \check{S}_{\lfloor ns \rfloor} |^2 \right). 
\end{align*}
Denoting as usual by $(\check{M}_n)$ the martingale $(\bc{a}_n\check{S}_n)_{n \geq 0}$, notice that by (\ref{formula:TheQuadraticVariationBalanced}), the remark that follows,  and (\ref{formula:asyptoticSeriesCounterbalanced}),  
\begin{equation*}
    \EE(\check{M}_n^2) = \EE(\langle \check{M},\check{M} \rangle_n ) \leq c \sum_{k=1}^n \bc{a}_k^2 \leq c n^{1+2p}
\end{equation*}
for some constant $c$ that might change from one inequality to the other. We deduce by Doob's inequality 
\begin{align*}
    \mathbb{P} \left(\sup_{s \in [2^{-(i+1)} \delta , 2^{-i} \delta  ]} \frac{|\check{S}_{\lfloor ns \rfloor}|}{\sqrt{n}}  > \epsilon \right) 
    &\leq 
    c \left( 2^{-i } \delta n\right)^{1+2p}   \frac{1}{\epsilon^2 \cdot  n \bc{a}^2_{\lfloor 2^{-(i+1)}  \delta n \rfloor}} \\
    & = c 2^{-i(1+2p)} \delta^{1+2p} \frac{n^{2p}}{\bc{a}^2_{\lfloor 2^{-(i+1)}  \delta n \rfloor}} \\
\end{align*}
which, recalling the asymptotic behaviour $\bc{a}_n \sim n^p$, yields for some constant $c$ that might differ from one line to the other: 
\begin{align*}
    \sup_n  \mathbb{P} \left(\sup_{s \in [2^{-(i+1)} \delta , 2^{-i} \delta  ]} \frac{|\check{S}_{\lfloor ns \rfloor}|}{\sqrt{n}}  > \epsilon \right) 
    &\leq c 2^{-i(1+2p)} \delta^{1+2p} \cdot  2^{2p(i+1)}\delta^{-2p} \\
    &= c \cdot 2^{-i} \delta. 
\end{align*}
From the previous estimate, we deduce the uniform bound 
\begin{align} \label{equation:counterbalanceOriginBound}
    \mathbb{P} \left( \sup_{s \in [0,\delta ] }\frac{|\check{S}_{\lfloor ns \rfloor} |}{\sqrt{n}} > \epsilon  \right)
    & \leq \sum_{i=0}^\infty  \mathbb{P} \left(\sup_{s \in [2^{-(i+1)} \delta , 2^{-i} \delta  ]} \frac{|\check{S}_{\lfloor ns \rfloor}|}{\sqrt{n}}  > \epsilon \right)  \nonumber \\
    & \leq \sum_{i=0}^\infty \sup_n \mathbb{P} \left(\sup_{s \in [2^{-(i+1)} \delta , 2^{-i} \delta  ]} \frac{|\check{S}_{\lfloor ns \rfloor}|}{\sqrt{n}}  > \epsilon \right)  
    \leq  K \cdot \delta. 
\end{align}
Finally, write $X^{(n)} = (\frac{1}{\sqrt{n}} \check{S}_{\lfloor nt \rfloor})_{t \in \mathbb{R}^+}$. Since for any $\delta >0$ we have $(X^{(n)}_t)_{t \geq \delta} \Rightarrow (\check{B}_{t})_{t \geq \delta}$ as $n \uparrow \infty$ and of course $(\check{B}_{t +\delta })_{t \in \mathbb{R}^+} \Rightarrow (\check{B}_{t })_{t \in \mathbb{R}^+}$ as $\delta \downarrow 0$, we deduce that there exists some decreasing  sequence $(\delta(n)) \downarrow 0$ such that 
\begin{equation*}
    \left( X^{(n)}_{s+ \delta(n)} \right)_{s \in \mathbb{R}^+} \Rightarrow \check{B} \quad \quad \text{ as }n \uparrow \infty
\end{equation*}
while by (\ref{equation:counterbalanceOriginBound}), 
\begin{equation*}
    \sup_{s \in [0,\delta(n)]} X^{(n)}_s \rightarrow 0 \quad \quad \text{ in probability}.
\end{equation*}
This establishes that the convergence $\left(\frac{1}{\sqrt{n}} \check{S}_{\lfloor nt \rfloor }\right)_{t \in \mathbb{R}^+} \Rightarrow \check{B}$ holds on  $\mathbb{R}^+$ and with this, we conclude our proof of Theorem \ref{thm:ConvergenceTripletDifusive}.

\begin{rem}
In the process of proving Theorem \ref{thm:ConvergenceTripletDifusive} in Section \ref{section:proofBoundedSteps} and \ref{section:reductionArgument},  we showed also that if we no longer consider the noise-reinforced random walk, we can extend the convergence of the pair to $p \in (0,1)$,
\begin{equation} 
    \left( \frac{1}{\sigma \sqrt{n}} {S}_{\lfloor nt \rfloor},  \frac{1}{\sigma \sqrt{n}} \check{S}_{\lfloor nt \rfloor}  \right)_{t \in \mathbb{R}^+} \Longrightarrow \left(  B_t ,  \int_0^t s^{p}d \beta^c_s \right)_{t \in \mathbb{R}^+} 
\end{equation}
 where as usual $\beta^c$, $B$ are two Brownian motions with $\langle B,\beta^c \rangle_t = (1-p) t$, and that the result still holds if $p = 1$ if we assume $X$ follows the Rademacher distribution, in which case the processes are independent. This is precisely the content of Theorem \ref{theorem:InvarianceCounterBalancedCase}. Finally, Theorem \ref{theorem:InvarianceDiffusiveRegime} also follows by recalling that  $\hat{S}_n  - n \EE(X)$ is a centred positive step-reinforced random walk and hence falls in our framework. 
\end{rem}

\section{The critical regime for the positive-reinforced case: proof of Theorem \ref{theorem:InvarianceCriticalRegime}} \label{section:criticalRegime}
In this last section we turn our attention to the critical regime $p=1/2$ for the noise reinforced case  and prove the invariance principle with our martingale approach. The arguments are very similar and rely on  exploiting the  martingale defined in Proposition \ref{prop:MultMartStart}, the MFCLT and a truncation argument. The main difference comes from the fact that, for $p=1/2$, the asymptotic behaviour of $\sum_{k=1}^n \bh{a}_k^{2}$ is no longer the one claimed in (\ref{equation:asymptSeries}). Namely, as we pointed out previously, 
\begin{equation*}
    \lim_{n \rightarrow \infty} \frac{1}{\log(n)} \sum_{k=1}^n \bh{a}_k^2 = 1 
\end{equation*}
and the different scaling that we will use  makes impossible to couple the convergence with the random walk or the counterbalanced random walk. Once again, we start with a law of large numbers-type result:
\begin{lem} \label{thm:LLNcritical}
Suppose $\|X\|_\infty < \infty$. We have the almost sure convergence
\begin{align*}
    \lim_{n \to \infty} \frac{\hat{S}_n}{\sqrt{n} \log n}=0 \quad \text{a.s.}
\end{align*}
and fortiori we  have $\lim_{n \to \infty} n^{-1} {\hat{S}_n}=0 \quad \text{a.s.}$
\end{lem}
\begin{proof}
The proof of this statement follows along the same lines as the proof of Lemma  \ref{thm:LLNDiffusive}. Since $p=1/2$ we have now, with the notation introduced in (\ref{notation:vn}), that as $n \to \infty$,
\begin{align*}
    \nu_n \sim  K' \cdot \log n, 
\end{align*}
where $K'$ is a positive constant. That is, $\nu_n$ increases slowly to infinity with a logarithmic speed. We obtain again from Theorem 1.3.24 in \cite{duflo} that 
\begin{align*}
    \frac{\hat{M}_n^2}{\log n}=O ( \log \log n) \quad \text{a.s.}
\end{align*}
Hence, as $\hat{M}_n= \bh{a}_n \hat{S}_n$, the above readily implies that 
\begin{align*}
    \bh{a}_n^2 \frac{\hat{S}_n^2}{ \log n}=O( \log \log n) \quad \text{a.s.}
\end{align*}
Further, we deduce from (\ref{equation:asymptSeriesCritical}) that for $p=1/2$,  
$
    \lim_{n \to \infty} \bh{a}_n^2 \cdot {n}=  1
$
 and hence we deduce that 
\begin{align*}
    \frac{\hat{S}_n^2}{n \log n} = O( \log \log n) \quad \text{a.s.}
\end{align*}
which immediately implies the claim.
\end{proof}
We now prove the invariance principle under the assumption of boundedness for $X$.
\begin{proof}[Proof of Theorem \ref{theorem:InvarianceCriticalRegime} when $\|X\|_\infty < \infty$]
The proof relies on similar ideas to the ones used in the proof of Theorem  \ref{theorem:InvarianceDiffusiveRegime}. Recalling that, 
\begin{equation*}
    \bh{a}_k \sim k^{-p}  = k^{-1/2}  \quad \quad \text{ as } k \rightarrow \infty
\end{equation*}
from the substitution $k = \lfloor n^t \rfloor$, we deduce that 
\begin{equation*}
    \bh{a}_{\lfloor n^t \rfloor} \sim \frac{1}{n^{t/2}} .
\end{equation*}
  Then, the limit (\ref{equation:InvarianceCriticalRegime}) can equivalently be shown  by establishing the desired convergence towards $B=(B_t)_{t \geq 0}$ for the following  sequence of martingales: 
\begin{equation*}
    \left( \frac{\bh{a}_{\lfloor n^t \rfloor} }{ \sqrt{\log(n)}}  \hat{S}_{\lfloor n^t \rfloor} \right)_{t \in \mathbb{R}^+} \Longrightarrow ( \sigma  B_t)_{t \in \mathbb{R}^+}.
\end{equation*}
Once again, we denote 
\begin{equation*}
     (\hat{N}^n_t)_{t \in \mathbb{R}^+} = \left( \frac{\bh{a}_{\lfloor n^t \rfloor} }{  \sqrt{\log(n)}}  \hat{S}_{\lfloor n^t \rfloor} \right)_{t \in \mathbb{R}^+} 
\end{equation*}
and deduce as before that for each $n \in \mathbb{N}$, the  predictable quadratic variation of $\hat{N}^n$ is given by
\begin{align}
\label{formula:quadVariationCritialCase}
 & \langle \hat{N}^{(n)}, \hat{N}^{(n)} \rangle_t = \notag \\
 & \hspace{5mm} \frac{1}{ \log(n)} \left( \sigma^2+  \sigma^2  (1-p) \sum_{k=2}^{\lfloor n^t \rfloor} \bh{a}_{k}^2 
     - p^2 \sum_{k=2}^{\lfloor n^t \rfloor} \bh{a}_{k}^2 \left( \frac{\hat{S}_{k-1} }{k-1} \right)^2 + p \sum_{k=2}^{\lfloor n^t \rfloor} \bh{a}_{k}^2 \left( \frac{{\hat{V}}_{k-1} }{k-1}\right)  \right).
 \end{align}
By the MFCLT, in order to prove our claim it suffices to show that 
\begin{equation*}
    \lim_{n \rightarrow \infty} \langle \hat{N}^{(n)}, \hat{N}^{(n)} \rangle_t =  \sigma^2 t \quad \quad \text{ a.s.}
\end{equation*}
and that $\sup_t |\Delta \hat{N}^{(n)}_t| \rightarrow 0$ in probability as $n \rightarrow \infty$.  Since $\|X\|_\infty < \infty$, this last requirement follows from very similar arguments to the ones we used in the proof of Theorem \ref{theorem:InvarianceDiffusiveRegime}. 
On the other hand, since $\log(\lfloor n^t \rfloor) / \log(n) \rightarrow t$ as $n \rightarrow \infty$ and by (\ref{equation:asymptSeriesCritical}), the first nontrivial term of (\ref{formula:quadVariationCritialCase}) satisfies the following convergence
\begin{equation*}
    \lim_{n \rightarrow \infty} \frac{1}{ \log(n)}  (1-p) \sum_{k=1}^{\lfloor n^t \rfloor} \bh{a}_{k}^2  = t(1-p).
\end{equation*}
By the same arguments we used in the proof of Theorem \ref{theorem:InvarianceDiffusiveRegime} but using the law of large numbers for the critical regime (Lemma \ref{thm:LLNcritical}), we obtain that the second term in (\ref{formula:quadVariationCritialCase}) converges to zero while for the last term, 
\begin{equation*}
     \lim_{n \rightarrow \infty} \frac{1}{l \log(n)} p \sum_{k=1}^{\lfloor n^t \rfloor} \bh{a}_{k}^2 \frac{{\hat{V}}_{k} }{k} = t \cdot p { \sigma^2 } \quad \text{ a.s.}
\end{equation*}
It follows that $\langle \hat{N}^{(n)} , \hat{N}^{(n)} \rangle_t \rightarrow t { \sigma^2 }$ for each $t$ as $n \rightarrow \infty$, which proves the desired result under the additional assumption that $\|X\|_\infty < \infty$.
\end{proof}
Now we establish the general case by means of the usual reduction argument. We will not be as detailed as before, since the ideas are exactly the same. We do still assume without loss of generality that the steps are centred.  
\begin{proof}[Proof of Theorem \ref{theorem:InvarianceCriticalRegime}, general case]
Maintaining the notation introduced for the truncated reinforced random walks of Section \ref{section:reductionArgument} as well as for the respective variances $\eta_K$ and $\sigma_K$ for $K > 0$,   Theorem \ref{theorem:InvarianceCriticalRegime} in the bounded step case shows for each $K >0$ the convergences in distribution as $n$ tends to infinity in the sense of Skorokhod, 
\begin{align}
    \left( \frac{\hat{S}^{\leq K}( \lfloor n^t \rfloor)}{\sqrt{ \log (n) n^t}} \right)_{t \in \mathbb{R}^+} & \implies \left( \sigma_K B(t)\right)_{t \in \mathbb{R}^+}   \label{convergence:Criticalbounded}
\end{align}
and from  $\lim_{K \to \infty} \sigma_K = \sigma$, it follows readily from   (\ref{convergence:Criticalbounded})  and the same arguments as before that  as $n$ tends to infinity, 
\begin{align*}
    \left(  \frac{\hat{S}^{\leq K(n))}( \lfloor n^t \rfloor)}{\sqrt{ \log (n) n^t}} \right)_{t \in \mathbb{R}^+} & \implies (\sigma {B}(t)  )_{t \in \mathbb{R}^+}
\end{align*}
for some increasing sequence $(K(n))_{n \geq 0}$ of positive real numbers converging towards infinity. On the other hand, from Lemma \ref{lem:maximalIneq} for $p=1/2$ we deduce that 
\begin{equation*}
      \lim_{n \to \infty} \frac{1}{n^t \log (n)} \EE \left( \sup_{k \leq n^t} | \hat{S}^{>b(n)} (k)|^2 \right) \leq  c_2 \lim_{n \to \infty}  \eta_{K(n)}^2 t =0 \label{ConditionCritical}
\end{equation*}
and from here we can  proceed as we did in the previous section. With this, we conclude the proof of  Theorem \ref{theorem:InvarianceCriticalRegime}.
\end{proof}

{
\textit{Acknowledgement: We warmly thank Jean Bertoin and Erich Baur for the all the fruitful discussions and feedback, as well as for introducing us to the subject.}
}
\bibliographystyle{plainurl} 
\bibliography{bib}

\begin{thebibliography}{10}

\bibitem{GaussianAproximation_BaiHuZhang}
Z.~D. Bai, Feifang Hu, and Li-Xin Zhang.
\newblock Gaussian approximation theorems for urn models and their
  applications.
\newblock {\em The Annals of Applied Probability}, 12(4):1149--1173, 2002.

\bibitem{BaurClass}
Erich Baur.
\newblock On a class of random walks with reinforced memory.
\newblock {\em Journal of Statistical Physics}, 181(3):772--802, 2020.
\newblock \href {https://doi.org/10.1007/s10955-020-02602-3}
  {\path{doi:10.1007/s10955-020-02602-3}}.

\bibitem{BaurBertoin}
Erich Baur and Jean Bertoin.
\newblock Elephant random walks and their connection to {P}\'olya-type urns.
\newblock {\em Phys. Rev. E}, 94, 2016.
\newblock \href {https://doi.org/10.1103/PhysRevE.94.052134}
  {\path{doi:10.1103/PhysRevE.94.052134}}.

\bibitem{Bercu}
Bernard Bercu.
\newblock A martingale approach for the elephant random walk.
\newblock {\em Journal of Physics A: Mathematical and Theoretical}, 51(1),
  2017.
\newblock \href {https://doi.org/10.1088/1751-8121/aa95a6}
  {\path{doi:10.1088/1751-8121/aa95a6}}.

\bibitem{BercuLaulinCenter}
Bernard Bercu and Lucile Laulin.
\newblock On the center of mass of the elephant random walk.
\newblock {\em Stochastic Processes and their Applications}, 2020.
\newblock \href {https://doi.org/10.1016/j.spa.2020.11.004}
  {\path{doi:10.1016/j.spa.2020.11.004}}.

\bibitem{Bertenghi}
Marco Bertenghi.
\newblock Functional limit theorems for the multi-dimensional elephant random
  walk.
\newblock 2020.
\newblock \href {http://arxiv.org/abs/2004.02004} {\path{arXiv:2004.02004}}.

\bibitem{BertenghiAsymptotic}
Marco Bertenghi.
\newblock Asymptotic normality of superdiffusive step-reinforced random walks,
  2021.
\newblock \href {http://arxiv.org/abs/2101.00906} {\path{arXiv:2101.00906}}.

\bibitem{BertoinCounterbalancing}
Jean Bertoin.
\newblock Counterbalancing steps at random in a random walk.
\newblock {\em arXiv preprint arXiv:2011.14069}, 2020.

\bibitem{BertoinNoise}
Jean Bertoin.
\newblock Noise reinforcement for {L}évy processes.
\newblock {\em Annales de l'Institut Henri Poincaré, Probabilities et
  Statistiques, 56(3):2236-2252}, 2020.

\bibitem{BertoinScalingExponents}
Jean Bertoin.
\newblock Scaling exponents of step-reinforced random walks.
\newblock {\em Probability Theory and Related Fields}, 2021.
\newblock \href {https://doi.org/10.1007/s00440-020-01008-2}
  {\path{doi:10.1007/s00440-020-01008-2}}.

\bibitem{BertoinUniversality}
Jean Bertoin.
\newblock Universality of noise reinforced {B}rownian motions.
\newblock {\em Progress in Probability, vol 77. Birkhäuser}, 2021.
\newblock \href {https://doi.org/doi.org/10.1007/978-3-030-60754-8_7}
  {\path{doi:doi.org/10.1007/978-3-030-60754-8_7}}.

\bibitem{Businger}
Silvia Businger.
\newblock The shark random swim.
\newblock {\em Journal of Statistical Physics}, 172(3):701--717, 2018.
\newblock \href {https://doi.org/10.1007/s10955-018-2062-5}
  {\path{doi:10.1007/s10955-018-2062-5}}.

\bibitem{GavaSchuetzColetti}
Cristian~F. Coletti, Renato Gava, and Gunter~M. Schütz.
\newblock Central limit theorem and related results for the elephant random
  walk.
\newblock {\em Journal of Mathematical Physics}, 58, 2017.
\newblock \href {https://doi.org/10.1063/1.4983566}
  {\path{doi:10.1063/1.4983566}}.

\bibitem{Coletti2017}
Cristian~F Coletti, Renato Gava, and Gunter~M Schütz.
\newblock A strong invariance principle for the elephant random walk.
\newblock {\em Journal of Statistical Mechanics: Theory and Experiment},
  2017(12):123207, dec 2017.
\newblock \href {https://doi.org/10.1088/1742-5468/aa9680}
  {\path{doi:10.1088/1742-5468/aa9680}}.

\bibitem{Coletti2019}
Cristian~F Coletti and Ioannis Papageorgiou.
\newblock Asymptotic analysis of the elephant random walk.
\newblock {\em Journal of Statistical Mechanics: Theory and Experiment},
  2021(1):013205, 2021.
\newblock \href {https://doi.org/10.1088/1742-5468/abcd36}
  {\path{doi:10.1088/1742-5468/abcd36}}.

\bibitem{duflo}
Marie Duflo.
\newblock {\em Random iterative models}, volume~34.
\newblock Springer Science \& Business Media, 2013.

\bibitem{Gonzales}
Manuel Gonz{\'a}lez-Navarrete and Rodrigo Lambert.
\newblock Non-markovian random walks with memory lapses.
\newblock {\em Journal of Mathematical Physics}, 59(11):113301, 2018.
\newblock \href {https://doi.org/10.1063/1.5033340}
  {\path{doi:10.1063/1.5033340}}.

\bibitem{GonzalesMult}
Manuel González-Navarrete.
\newblock Multidimensional walks with random tendency.
\newblock {\em Journal of Statistical Physics volume}, 181:1138--1148, 2020.
\newblock \href {https://doi.org/10.1007/s10955-020-02621-0}
  {\path{doi:10.1007/s10955-020-02621-0}}.

\bibitem{GuevaraERW}
V{\i}ctor Hugo~V{\'a}zquez Guevara and Hugo~Cruz Su{\'a}rez.
\newblock An elephant random walk based strategy for improving learning
  (preprint).
\newblock \href {https://doi.org/10.13140/RG.2.2.10920.72960}
  {\path{doi:10.13140/RG.2.2.10920.72960}}.

\bibitem{Processes}
Jean Jacod and Albert~N. Shiryaev.
\newblock {\em Limit Theorems for Stochastic Processes}.
\newblock Springer, 2003.
\newblock \href {https://doi.org/10.1007/978-3-662-05265-5}
  {\path{doi:10.1007/978-3-662-05265-5}}.

\bibitem{JansonLimitBranchingProcesses}
Svante Janson.
\newblock Functional limi theorems for multitype branching processes and
  generalized polya urns.
\newblock {\em Stochastic processes and their Applications}, 110:177--245,
  2004.
\newblock \href {https://doi.org/10.1016/j.spa.2003.12.002}
  {\path{doi:10.1016/j.spa.2003.12.002}}.

\bibitem{KubotaTakei}
Naoki Kubota and Masato Takei.
\newblock Gaussian fluctuation for superdiffusive elephant random walks.
\newblock {\em Journal of Statistical Physics 177}, pages 1157--1171, 2019.
\newblock \href {https://doi.org/10.1007/s10955-019-02414-0}
  {\path{doi:10.1007/s10955-019-02414-0}}.

\bibitem{Kuersten}
R\"udiger K\"ursten.
\newblock Random recursive trees and the elephant random walk.
\newblock {\em Phys. Rev. E}, 93:032111, Mar 2016.
\newblock \href {https://doi.org/10.1103/PhysRevE.93.032111}
  {\path{doi:10.1103/PhysRevE.93.032111}}.

\bibitem{SchuetzTrimper}
Gunter~M. Sch\"utz and Steffen Trimper.
\newblock Elephants can always remember: Exact long-range memory effects in a
  non-{M}arkovian random walk.
\newblock {\em Phys. Rev. E}, 70, 2004.
\newblock \href {https://doi.org/10.1103/PhysRevE.70.045101}
  {\path{doi:10.1103/PhysRevE.70.045101}}.

\bibitem{WhittFCLT}
Ward Whitt.
\newblock Proofs of the martingale {FCLT}.
\newblock {\em Probab. Surveys}, 4:268--302, 2007.
\newblock \href {https://doi.org/10.1214/07-PS122}
  {\path{doi:10.1214/07-PS122}}.

\end{thebibliography}
\end{document}